\documentclass[12pt,leqno]{amsart}
\usepackage{palatino}
\usepackage{amsmath}
\usepackage{amssymb}
\usepackage{mathtools}
\usepackage{amsmath}
\usepackage{amssymb}
\usepackage{latexsym}
\usepackage{amsopn}
\usepackage{amsthm}
\usepackage{url}
\usepackage{multicol}
\usepackage[margin=1.5in]{geometry}
\usepackage{graphicx}
\graphicspath{ {./images/}}
\def\underTilde#1{{\baselineskip=0pt\vtop{\hbox{$#1$}\hbox{$\sim$}}}{}}

\usepackage{tikz}
\usetikzlibrary{positioning}

\def\underTilde#1{{\baselineskip=0pt\vtop{\hbox{$#1$}\hbox{$\sim$}}}{}}

\def\bSigma{\underTilde{\Sigma}}
\def\bPi{\underTilde{\Pi}}
\def\bDelta{\underTilde{\Delta}}

 \newtheorem{thm}{Theorem}[section]
 \newtheorem{cor}[thm]{Corollary}
 \newtheorem{lem}[thm]{Lemma}

\newtheorem*{theorem}{Theorem}

\newtheorem*{theorem1}{Theorem 1}
\newtheorem*{theorem2}{Theorem 2}
\newtheorem*{theorem3}{Theorem 3}
\newtheorem*{theorem4}{Theorem 4}
\newtheorem*{theorem5}{Theorem 5}
\newtheorem*{theorem6}{Theorem 6}

 \newtheorem{prop}[thm]{Proposition}
  
 \theoremstyle{definition}
 \newtheorem{defn}[thm]{Definition}
 \newtheorem{ex}[thm]{Example}
 \newtheorem{rem}[thm]{Remark}

 \theoremstyle{remark}

\newtheorem{claim}[thm]{Claim}

 \numberwithin{equation}{section}
\begin{document}

\title{Hierarchies of Beliefs and Measurable Uniformizations}

\author{Stuart Zoble \\ December 1, 2025}

\thanks{{\em Acknowledgments}  We'd like to thank Adam Brandenburger for many illuminating discussions about Game Theory.  We also thank Hugh Woodin and Paul Larson for helpful comments.}

\address{Department of Mathematics, Albert Nerkin School of Engineering, The Cooper Union for the Advancement of Science and Art, New York City}
\email{szoble@gmail.com \\ stuart.zoble@cooper.edu}


\begin{abstract}  We extend the Fundamental Theorem of Epistemic Game Theory to games with Baire class one payoffs and locally compact Polish strategy spaces, and under Projective Determinacy, to games with analytically measurable payoffs and arbitrary Polish strategy spaces.  We show that in full generality, the statement that rationalizable strategies are consistent with rationality and common belief in rationality follows from the Axiom of Real Determinacy, has a characterization in terms of real Gale-Stewart games, fails under the Continuum Hypothesis, and in the framework of interactive epistemics is equivalent to the Measurable Uniformization Principle from the Solovay model.   
\end{abstract}

%%% ----------------------------------------------------------------------
\maketitle
%%% ----------------------------------------------------------------------
%\small
%\tableofcontents
%\normalsize

%\begin{multicols}{2}

\section{Introduction and Statement of Results}  The Fundamental Theorem of Epistemic Game Theory reconciles explicit and implicit approaches to formalizing an intuitive notion of rationality in games - namely, the idea that each player chooses a strategy to maximize their expected payoff according to some subjective belief about the other player's strategy choices, and forms this belief in a way that is consistent with the decision problem of each player being common knowledge.\footnote{For our purposes a (two-person strategic form) game is a pair $G = (\pi_1,\pi_2)$ of bounded functions $\pi_i: X_1 \times X_2 \rightarrow \mathbb{R}$ where each $X_i$ is a Polish space.  Each $\pi_i$ extends to $\Delta(X_1) \times \Delta(X_2)$, where $\Delta(X_i)$ is the space of Borel probability measures on $X_i$, by taking expectations (assuming $\pi_i$ is universally measurable).  An element of $\Delta(X_i)$ is referred to as a belief or mixed strategy. A strategy $s \in X_i$ is a best response to $\mu \in \Delta(X_j)$ if $\pi_i(s,\mu) \geq \pi_i(\bar{s},\mu)$ for every $\bar{s} \in X_i$.  In a context where $i$ is used to index a player we often use the index $j$ to represent the other player.}

The epistemic framework employs type structures\footnote{A concept originating in work of Harsanyi \cite{harsanyi}; see \cite{pac} for a philosophical discussion.} designed to explicitly model the idea that strategy choices should be rational in the sense that each player maximizes according to a belief, that each player believes the other is choosing rationally, that each player believes the other player believes this, and so on.  For our purposes, type structures are Polish spaces $T_i$, called type spaces, and homeomorphisms 
$$\lambda_i: T_i \rightarrow \Delta(X_j \times T_j)$$ where $X_i$ is the Polish strategy space of player $i$.\footnote{We will use the construction of the universal type space in \cite{brand2} in which types are explicitly coherent hierarchies of beliefs; see Definition 4.1.}  Thus a type for player $i$ is identified with a belief (probability measure) over the strategies and types of player $j$.  We say a strategy $s \in X_i$ for player $i$ is rational at $t \in T_i$ if $s$ is a best response to $\mbox{marg}_{X_j}(\lambda_i(t))$, and we let $R_i^1$ denote the set of such pairs $(s,t)$ for player $i$.  We say that $s$ satisfies second order rationality at $t$ if $s$ is rational at $t$ and $\lambda_i(t)$ concentrates on the set of pairs in $X_j \times T_j$ at which player $j$ is rational, and so on.  Setting $$R^{n+1}_i = \{ (s,t) \in R^{n}_i \ | \ \lambda_i(t)( R^{n}_j) = 1\}$$ for each player $i$ and $n > 0$, we say that a strategy $s \in X_i$ is consistent with Rationality and Common Belief in Rationality, abbreviated RCBR, if there is a type $t \in T_i$ such that $s$ satisfies $n^{th}$-order rationality at $t$ for every $n > 0$, that is $(s,t) \in R_i^{n}$ for every $n$.  Thus, RCBR requires a witness $t$, which we can think of as encoding a hierarchy of beliefs, justifying the strategy as a viable choice.

The implicit or strategic approach, on the other hand, involves various procedures for iteratively eliminating irrational or unjustified strategies from a game.  The argument is that both players are rational so irrational strategies can be dispensed with.  The players believe this about each other, so strategies that are irrational in the reduced spaces can be similarly dispensed with, and so on until the process converges.   For example, the Rationalizable strategies are obtained by iteratively eliminating never-best-responses strategies.  If $X_i^{n}$ are the strategies remaining for player $i$ after n rounds then $$X_i^{n+1} = \{ s \in X_i^n \ | \ s \mbox{ is a best response to some } \mu \in \Delta(X_j^n)\},$$ and the rationalizable strategies are those that survive every elimination round.  It should be noted that in non-compact games the process can require trans-finitely many rounds before converging, a fact first observed by Lipman in \cite{lipman}.\footnote{The concept of Rationalizability has its orgin in Bernheim \cite{bern} and Pearce \cite{pearce}.  The latter writes that the  solution concept is, "...based on three assumptions: ASSUMPTION (Al): When a player lacks an objective probability distribution over another player's choice of strategy, he forms a subjective prior that does not contradict any of the information at his disposal. ASSUMPTION (A2): Each player maximizes his expected utility relative to his subjective priors regarding the strategic choices of others. ASSUMPTION (A3): The structure of the game (including all participants' strategies and payoffs, and the fact that each player satisfies Assumptions (A1) and (A2)) is common knowledge..."}  For another example, the Iteratively Undominated or IU strategies are obtained in similar fashion by removing strategies at each stage that are strictly dominated.\footnote{A strategy $s \in X_i$ is strictly dominated if there is $\mu \in \Delta(X_i)$ if $\pi_i(s,t) < \pi_i(\mu,t)$ for every $t \in X_j.$}

The fundamental theorem holds that in finite games (or games with compact strategy spaces and continuous payoffs), the two approaches coincide, specifically the rationalizable strategies, either in the more or less stringent sense (best response relative to all alternatives or relative to remaining alternatives)\footnote{For every $s \in X_i$ \ $\pi(s,\mu) \geq \pi(\bar{s},\mu)$, or for every $s \in X_i^n$ \ $\pi(s,\mu) \geq \pi(\bar{s},\mu)$.}, and the IU strategies\footnote{See Prop. 2.3 for a proof that never-best-response strategies are strictly dominated in the compact and continuous case (the reverse implication is immediate).  This often cited part of the equivalence holds for finite games but does not generalize; see \cite{cz}.}, both have an epistemic basis in terms of RCBR.  In \cite{brand}, {\em The Language of Game Theory}, Brandenburger writes:\footnote{See also "Theorem 1 (Fundamental Theorem of Epistemic Game Theory)" of \cite{perea}, as well as the handbook article \cite{dekel} and the book \cite{perea2}.}

%\footnote{"Epistemic game theory formalizes assumptions about rationality and mutual beliefs in a formal language, then studies their behavioral implications in games. Specifically, it asks: what do different notions of rationality and different assumptions about what players believe about...what others believe about the rationality of players imply regarding play in a game? A well-known example is the equivalence between common belief in rationality and iterated deletion of dominated strategies." ---Eddie Dekel, Marciano Siniscalchi, Epistemic Game Theory, Handbook of Game Theory with Economic Applications \cite{dekel} }

\begin{quote}  "Given a game and an IU strategy for each player, we can build a belief model so that these strategies are consistent with RCBR in the model.  With this, we can say that the epistemic condition of RCBR not only implies that IU strategies will be be played, but actually identifies this set (and no smaller set) of strategies. We get a characterization of the epistemic condition of RCBR.  This characterization result - which we call the Fundamental Theorem of EGT [epistemic game theory] - has been proved a number of times in a variety of different forms." 
\end{quote} 

Note that if $R_i$ denotes the set of strategies for player $i$ which are consistent with RCBR, and $s \in R_i$, then $s$ is a best response to a belief $\mu$ with $\mu(R_j) = 1$.  The strategy rectangle $R_i \times R_j$ is said to have the best response property and is therefore easily seen by induction to survive the procedures of the implicit approaches (in fact the rationalizable strategies are characterized as the maximal strategy rectangle with the best response property).  Thus, the non-trivial direction in the fundamental theorem involves showing that strategies that survive the elimination procedures have a type witnessing RCBR.  

\begin{theorem} (Fundamental Theorem of EGT) The Rationalizable strategies and the Iteratively Undominated strategies, either in the more or less stringent sense, coincide with the strategies which are consistent with Rationality and Common Belief in Rationality in games with continuous payoffs and compact Polish strategy spaces.\footnote{This result has its origins in Tan and Werlang \cite{tw} and Brandenburger and Dekel \cite{brand4} (see the comments by Perea in \cite{perea}), but the statement doesn't seem to appear in the early literature in this form.  We prove this in Theorem 2.10 below; see also Theorem 3.6 for a more abstract version.}
\end{theorem}

Our motivation here is to extend this theorem in the setting of infinite (more precisely, non-compact) games.  As we pass to infinite strategy spaces there is a divergence in the implicit approaches on complexity grounds alone.  In the case of integer games, as the prototypical non-compact case, while the epistemic notion is analytic, the solution concepts based on iterative deletion procedures (specifically those that are memoryless - that is, relative to surviving alternatives - or those based on notions of dominance) are generally neither analytic nor co-analytic with the exception of the notion of Rationalizability in the more stringent sense (see \cite{cz} for this analysis).  With respect to Rationalizability in the more stringent sense, henceforth just Rationalizability, Arieli established the equivalence with RCBR in the category of continuous games in \cite{arieli}.

\begin{theorem} (Arieli) Rationalizability is equivalent to Rationality and Common Belief in Rationality for any game with continuous bounded payoffs and Polish strategy spaces. \end{theorem}

We will generalize this equivalence in the abstract setting of what we call interactive epistemics, with relations between strategies and beliefs of the form $$\text{E}_i \subseteq X_i \times \Delta(X_j)$$ as the primary object of study, and with games as the motivating application via the best response relation $$\text{E}^{G}_i = \{ (s,\mu) \ | \ s \mbox{ is a best response to } \mu \mbox{ in G}\} \subseteq X_i \times \Delta(X_j),$$ by defining the notions of E-Rationalizable strategies, $$\text{RAT}(\text{E}) \subseteq X_1 \times X_2,$$ and strategies consitent with "E and Common Belief in E" $$\text{RCBR}(\text{E}) \subseteq X_1 \times X_2,$$ for such relations.\footnote{When $\text{E} = \text{E}^{\text{G}}$ is the best response relation of a game $\text{G}$ we have $\text{RAT}(\text{G}) = \text{RAT}(\text{E})$ and $\text{RCBR}(\text{G}) = \text{RCBR}(\text{E})$ are the rationalizable and RCBR strategies of the game.}  

The Fundamental Theorem in this setting asserts that E-rationalizable strategies satisfy E and common belief in E when E is a pair of closed relations and the underlying Polish spaces are compact (see Theorem 3.6).   Arieli's theorem can be reformulated and understood as the analytic case in this framework.  We will prove that 
$$\text{RAT}(\text{E}) = \text{RCBR}(\text{E})$$ for $\text{E} \in \bSigma^1_1$, and for $\text{E} \in \bPi^1_2$ assuming $\bPi^1_3$-Determinacy, allowing us to extend the fundamental theorem to a wider class of games.\footnote{We will show that $\text{E}^{\text{G}}$ is Borel for Baire class one payoff functions on locally compact Polish spaces (Threorem 7.9), and that it can be $\bPi^1_1$-complete for Baire class 2 payoffs (Proposition 7.17).  We seem to need $\bPi^1_3$-determinacy even for Borel measurable games, as we need $\bPi^1_3$ to be a ranked class, and for $\bSigma^1_3$ relations to have measurable uniformizations (see Theorem 7.8 and Theorem 7.14). }

\begin{theorem1} Rationalizability is equivalent to Rationality and Common Belief in Rationality for any game with bounded Baire class one payoffs played on locally compact Polish strategy spaces.  \end{theorem1}

\begin{theorem2} Assuming $\bPi^1_3$-Determinacy, Rationalizability is equivalent to Rationality and Common Belief in Rationality for any game with bounded analytically measurable payoffs played on Polish strategy spaces.  Under Projective Determinacy, the equivalence holds for games with bounded projective payoffs played on Polish strategy spaces. 
\end{theorem2}

This equivalence can be understood in full generality as a kind of regularity property, and is fundamentally connected with consequences of Determinacy Axioms, specifically measurable uniformization.  Indeed, a determinacy argument, which we now sketch, yields a third equivalent characterization of the intuitive notion of rationality in terms of the existence of winning strategies in Gale-Stewart games.  

Given a strategy $s \in X_i$ of the underlying game, consider a real Gale-Stewart game $G_s$ where player $I$ attempts to justify $s$ by playing pairs $(\mu_{2n},b_{2n})$ at even stages, and player $II$ challenges player I by playing $s_{2n+1}$ at odd stages, under the requirements that $b_{2n}$ is Borel with $\mu_{2n}(b_{2n}) = 1$, $s_{2n+1} \in b_{2n}$ and $s_{2n+1}$ is a best response to $\mu_{2n+2}$ (with $s$ a best response to $\mu_0$ at the beginning), with player $II$ winning if player $I$ is unable to play at any point.  
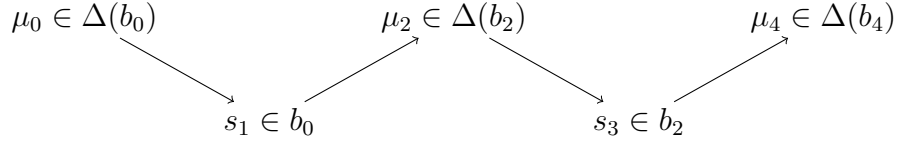
\begin{figure}[h]
\centering
    \begin{tikzpicture}[%
            node distance = 12mm,
            state/.style = {%
                %rectangle,
                        %draw = black,
                   inner sep = 0pt,
                minimum size = 5mm,
                thick,
            },
            beta/.style = {%
            node distance = 2mm,
                inner sep = 1pt,
            },
            auto,
        ]

        \node (a2) [state] {$\mu_0 \in \Delta(b_0)$};
        \node (a3) [below right=of a2,state] {$s_1 \in b_0$};
         \node (a4) [above right=of a3,state] {$\mu_2 \in \Delta(b_2)$};
                 \node (a5) [below right=of a4,state] {$s_3 \in b_2$};
         \node (a6) [above right=of a5,state] {$\mu_4 \in \Delta(b_4)$};
           \draw [->] (a2) -- (a3);
            \draw [->] (a3) -- (a4);
             \draw [->] (a4) -- (a5);
              \draw [->] (a5) -- (a6);

        \end{tikzpicture}
\caption{The Game $G_s$} 
\end{figure}

This game is quasi-determined with player $I$ having a winning quasi-strategy if and only if $s$ is Rationalizable.  Under Determinacy, player $I$ can be assumed to have a winning strategy and, as all sets are measurable, such a strategy can naturally be converted into a coherent hierarchy of beliefs (a type) witnessing that $s$ is consistent with RCBR. 

A more direct argument for the equivalence uses uniformization and measurability hypotheses which are consequences of $\mathsf{AD}_{\mathbb{R}}$.  As mentioned, the rationalizable strategies of a game can be characterized as the largest set of strategies with the property that each strategy for player $i$ in the set - call these $S_i \subseteq X_i$ - is a best response to a belief concentrating on the strategies for the other player in the set.  Given a uniformization $$\mu^i:S_i \rightarrow \Delta(S_j)$$ of this best response relation in strategy-belief space for each player $i$, we can lift, under measurability assumptions, to a map $$t^i:S_i \rightarrow T_i$$ which associates a type for player $i$ to each rationalizable strategy for that player in such a way that player $i$ can play a rationalizable strategy $s \in S_i$ believing at $t^i(s)$ that player $j$ plays some rationalizable strategy $\bar{s} \in S_j$, and has type given by $t^j(\bar{s})$, according to the distribution $\mu^i(s)$ over such $\bar{s}$.  In other words,  
$$\mbox{marg}_{X_j}(\lambda_i(s)) = \mu^i(s) \mbox{ and }
\lambda_i(t^i(s))(t^j) = 1$$ 
\noindent for each player $i$ and $s \in R_i$ 
where $t^j$ is identified with its graph.  This is sufficient for RCBR to hold for the Rationalizable strategies.

We will reduce the hypothesis for the argument above to the Measure Uniformization Principle $\mathsf{MUP}$ from the Solovay model (see Theorem 1 of \cite{solovay} and also \cite{rais}):

\begin{quote} ($\mathsf{MUP}$) A family $\{A_x \ | \ x \in B \}$ of nonempty sets of reals indexed by a set $B \subseteq \mathbb{R}$ of positive Lebesgue measure can be uniformized almost everywhere by a Borel function,
\end{quote}

\noindent which we show is equivalent to a measure extension axiom $\mathsf{MEA}$ abstracted from a paper of Lubin \cite{lubin}, used in the analytic case in \cite{arieli}:
\begin{quote} ($\mathsf{MEA}$) If $X,Y$ are Polish space, $A \subseteq X \times Y$ and $\mu \in \Delta(X)$ with $\mu(\mbox{proj}_{X}(A)) = 1$ then there is $\nu \in \Delta(X \times Y)$ with $\nu(A) = 1$ and $\mu = \mbox{marg}_X(\mu)$. 
\end{quote}

We show that this hypothesis is optimal in the setting of interactive epistemics, and entails an additional characterization in terms of winning strategies in real Gale-Stewart games.  A key point is that a measurable winning strategy in $G_s^{\text{E}}$ gives rise to a type that is a witness to $s \in \text{RCBR}(\text{E})$, whereas $s \in \text{RAT}(\text{E})$ is equivalent to the existence of a winning quasi-strategy.

\begin{theorem3}  The following consequences of $\mathsf{AD}_{\mathbb{R}}$ are equivalent.
\begin{enumerate}
    \item The Measure Uniformization Principle $\mathsf{MUP}$
    \item The Measure Extension Axiom $\mathsf{MEA}$
    \item $\text{RAT}(\text{E}) = \text{RCBR}(\text{E})$ for 
    for all relations $E$ between strategies and beliefs on underlying Polish spaces
    \item All sets are Lebesgue measurable, and for any relations $E$ between strategies and beliefs on underlying Polish spaces, a strategy $s$ belongs to $\text{RAT}_i(\text{E})$ if and only if player $I$ has a winning strategy the game $G^{\text{E}}_s$.
\end{enumerate}
\end{theorem3}
\begin{theorem4}  Under the Measure Uniformization Principle, Rationalizablity is equivalent to Rationality and Common Belief in Rationality in any game with bounded payoffs and Polish strategy spaces.  Both properties are characterized by player I having a winning strategy in the auxiliary Gale-Stewart game.
\end{theorem4}

Finally, we will give a counterexample under the Continuum Hypothesis, and conclude that the equivalence, even for universally measurable best-response relations, is independent of set theory without the axiom of choice.

\begin{theorem5}  Assuming the Continuum Hypothesis there is a game with Polish strategy spaces, bounded payoffs, and universally measurable best response relations with a rationalizable strategy that is not consistent with rationality and common belief in rationality.  \end{theorem5}
This hypothesis can be reduced to the existence of a non-measurable set if we put no restrictions on the game.  That is, under the assumption that rationalizable strategies are consistent with rationality and common belief in rationality in any game with bounded payoffs played on Polish strategy spaces, we show that all sets of reals are Lebesgue measurable.  

\begin{theorem6}  The statement that Rationalizablity is equivalent to Rationality and Common Belief in Rationality in any game with bounded payoffs and Polish strategy spaces is equiconsistent with the existence of an inaccessible cardinal.  \end{theorem6}

\section{Background and the Fundamental Theorem of EGT}  A basic framework for a decision problem consists of a strategy or choice space $X$, a space of uncertainty $Y$, and a bounded payoff function $$\pi:X \times Y \rightarrow \mathbb{R}.$$  We will always assume that $X$ and $Y$ are Polish.   We extend $\pi$ to $$\pi:\Delta(X) \times \Delta(Y) \rightarrow \mathbb{R}$$ by taking expectations, that is, $$\pi(\mu,\nu) = \int \int \pi \ d\mu \ d \nu,$$  
where $\Delta(X)$ and $\Delta(Y)$ denote the space of Borel probability measures on $X,Y$ respectively,  assuming $\pi$ is universally measurable.

\begin{defn}  We say a strategy $s \in X$ is a best response strategy if there is $\mu \in \Delta(Y)$ such that $$\pi(s,\mu) \geq \pi(t,\mu)$$
\noindent for every $t \in X$.  Otherwise we call $s$ a never-best-response strategy.  We say that a strategy $s \in X$ is strictly dominated if there is $\mu \in \Delta(X)$ such that $$\pi(s,y) < \pi(\mu,y)$$ \noindent for every $y \in Y$. \end{defn}
 
\begin{defn}   We call a pair $G = (\pi_1,\pi_2)$
a (two-person strategic form) game if each $\pi_i$ is a bounded real-valued function defined on $X_1 \times X_2$ for Polish spaces $X_1$ and $X_2$.  
Each $\pi_i$ extends to $\Delta(X_1) \times \Delta(X_2)$ by taking expectations (under the assumption that $\pi_i$ is universally measurable, which will always be the case in what follows).  The space $X_i$ is the strategy space for player $i$, and an element of $\Delta(X_i)$ is referred to as a belief or mixed strategy. In a context where $i$ is used to index a player, we often use the index $j$ to represent the other player (and we sometimes regard each $\pi_i$ as being defined on $X_i \times X_j$ as it is more convenient).  A pair $(\sigma_1,\sigma_2)$ with each $\sigma_i \in \Delta(X_i)$is a Nash equilibrium of a game $G$ if 
$$\pi_1(\sigma_1,\sigma_2) \geq \pi_1(\mu,\sigma_2) \mbox { and } \pi_2(\sigma_1,\sigma_2) \geq \pi_2(\sigma_1,\nu)$$
\noindent for every $\mu \in \Delta(X_1)$ and $ \nu \in \Delta(X_2)$. 
Finally, regarding $X_j$ and the space of uncertainty for player $i$, the notions from Definition 2.1 carry over to games.  \end{defn}

\begin{prop}  Suppose $X,Y$ are compact and $\pi \in C(X \times Y)$.  A strategy $s \in X$ is a never-best response for the decision problem $\pi$ if and only if $s$ is strictly dominated.\footnote{This notion traces back to \cite{arrow} and 3.2.1 of \cite{vandamme}. See also \cite{pearce} for the finite case, and \cite{zimper} for the compact case.} \end{prop}

\begin{proof}  Let $s_0 \in X$ be a strategy in the decision problem $\pi$ which is a never-best-response.  Here we adapt Pearce's argument (see Lemma 3 of \cite{pearce}) from the finite to show that $s_0$ must be strictly dominated.  We define an auxiliary game 
$(\bar{\pi}_{X},\bar{\pi}_{Y})$ as follows.  Let $\bar{pi}_{X}(s,t) = \pi(s,t) - \pi(s_0,t)$ so that 
$$\bar{\pi}_{X}(\sigma,\tau) = \pi_X(\sigma,\tau)-\pi_X(s_0,\tau)$$ for mixed strategies $\sigma \in \Delta(X)$ and $\tau \in \Delta(Y)$.  We set 
$$\bar{\pi}_{Y} = - \bar{\pi}_{X}.$$  
\noindent Now let $(\sigma,\tau)$ be a Nash equilibrium of this game, which exists by the Glicksberg generalization of the Kakutani Fixed Point Theorem (see section 2 of \cite{glicksberg}).  It is easy to see that $\sigma$ strictly dominates $s_0$.  For any $t \in Y$ we have $$\bar{\pi}_{X}(\sigma,t) \geq \bar{\pi}_{X}(\sigma,\tau)$$ by the Nash condition for the $Y$ player, and 
$$\bar{\pi}_{X}(\sigma,\tau) > \bar{\pi}_{X}(s_0,\tau) = 0$$ by the Nash condition for the $X$ player and assumption on $s_0$.  Thus $\pi(\sigma,t) > \pi(s_0,t)$ as desired. \end{proof}

We now define several solution concepts obtained by iteratively eliminating unjustified or irrational strategies.
On the one hand we consider strategies that survive trans-finite iterated deletion of never-best-responses relative to all alternatives, which we call the more stringent procedure, and those that survive iterated deletion of never-best-responses relative only to surviving alternatives, which we call the less stringent, or memory-less, procedure.  These are variations on the concept of Rationalizability, which has its origins in Bernheim \cite{bern} and Pearce \cite{pearce}. It was observed by Lipman in \cite{lipman} that such iterated deletion procedures could require trans-finitely many steps in order to converge in the non-compact setting.

\begin{defn}  For a game $G = (\pi_1,\pi_2)$, the respective successor-stage survival conditions for rationalizability in the more stringent sense ($\text{RAT}$) and in the less stringent sense ($\text{MRAT}$ with $\text{M}$ standing for "memory-less") are 

\begin{equation} \tag{\text{RAT}} 
\begin{multlined} X^{\alpha+1}_i = \{ s \in X^{\alpha}_i \ | \ \exists \mu \in \Delta(X_j)  \ \ \mu(X_j^{\alpha}) = 1   
\mbox{ and } \\ \forall t \in X_{i}^{0} \ \pi_i(s,\mu) \geq \pi_i(t,\mu) \}
\end{multlined}
\end{equation}
\begin{equation} \tag{\text{MRAT}} 
\begin{multlined} X^{\alpha+1}_i = \{ s \in X^{\alpha}_i \ |  \ \exists \mu \in \Delta(X_j)  \ \ \mu(X_j^{\alpha}) = 1  
\mbox{ and }  \\ \forall t \in X_{i}^{\alpha} \ \pi_i(s,\mu) \geq \pi_i(t,\mu) \}
\end{multlined}
\end{equation}
\noindent where $X_i = X_i^0$ is the starting strategy set for player $i$, and $j$ denotes the other player.
Intersections are taken at limit stages, $$X_i^{\alpha} = \bigcap_{\beta < \alpha} X_i^{\beta}.$$
\noindent If $\alpha$ is least such that $X_i^{\alpha+1} = X_i^{\alpha}$ for both $i$ we say $X^{\alpha}_i$ are the strategies that survive the procedure for player $i$ and define the associated solution concept $\text{C}$ (either $C = \text{RAT}$ or $C = \text{MRAT}$) by $$\text{C}(G)_i = X_{i}^{\alpha}.$$ 
We can use $\alpha(\text{C}(G))$ to denote this ordinal and $\text{C}(G)^{\beta}_i$ to denote the strategies for player $i$ which survive to stage $\beta$.  \end{defn}

\begin{rem}  We will identify any Borel measure $\mu \in \Delta(X)$ on a Polish space $X$ with its completion, defined on the $\sigma$-algebra generated by the Borel sets of $X$ together with the null sets of $\mu$, the so-called $\mu$-measurable sets (see 17A of \cite{kechris}). \end{rem}

\begin{rem} With respect to the general definition of rationalizability, a variation of our definition is used in \cite{arieli}.  There, the measure $\mu$ in Definition 2.8 is only required to be a Borel measure on the set $X_j^{\alpha}$, that is, $\mu \in \Delta(X_j^{\alpha})$ rather than a measure on $X_j$ with $\mu(X_j^{\alpha}) = 1$.  Such a measure $\mu \in \Delta(X_j^{\alpha})$ generates a unique measure $\hat{\mu} \in \Delta(X_j)$ (defined by $\hat{\mu}(B) = \mu(B \cap X_j^{\alpha})$) even if $X_j^{\alpha}$ is not $\mu$-measurable (the definitions are equivalent if $X_j^{\alpha}$ is $\mu$-measurable) and the results of this paper would go through with only cosmetic changes if we were to use this definition.  We have opted however to require $\mu(X_j^{\alpha}) = 1$ in the definition.  \end{rem}

\begin{rem} We say that $B_1 \times B_2$ has the best response property, where each $B_i \subseteq X_i$, if for each $i$ every $s \in B_i$ is a best response to some $\mu \in \Delta(X_j)$ with $\mu(B_j) = 1$ in the sense that $$\pi_i(s,\mu) \geq \pi_i(\bar{s},\mu)$$ \noindent for every $\bar{s} \in X_i$. Let $$\text{B}_i(G) = \bigcup \{ B_i \subseteq X_i \ | \ \exists B_j \subseteq X_j \ (B_i \times B_j \mbox{ has the best response property)} \},$$
\noindent Then $\text{B}_1(G) \times \text{B}_2(G)$ has the best response property itself, and thus is the maximal such set.  It is clear that $\text{B}_i(G) = \text{RAT}_i(G)$ for each $i$. Indeed, the set $\text{RAT}_1(G) \times \text{RAT}_2(G)$
has the best response property and strategies in $\text{B}_1(G) \cup \text{B}_2(G)$ survive every stage of elimination.
\end{rem}

\begin{ex}   Let $\leq$ be a pre-wellordering of a Polish space $X$ (for example $X = \omega$) of length $\alpha$ which is suitably measurable.  Set $X_1 = X_2 = X$ and define $$\pi_1(x,y) =
\begin{cases}
1  &  y < x\\
0 & \mbox{otherwise}   \\

\end{cases}
$$      Let $\pi_2(x,y) = \pi_1(y,x)$.  It is easy to see that $\alpha(\text{RAT}(G)) = \alpha$ where $G = (\pi_1,\pi_2)$.   Let $\psi:X \rightarrow \alpha$ be $<$-order preserving and surjective.  A straightforward induction shows that at every stage $\gamma < \alpha$ we have $$\text{RAT}_i^{\gamma} = \psi^{-1}(\alpha \setminus \gamma)$$ for each $i$.
\end{ex}

We also consider strategies that survive the iterated removal of strictly dominated strategies in either the more or the less stringent sense, the latter removing strategies at each stage that are strictly dominated relative to remaining strategies for the opponent by a mixed strategy that concentrates only on remaining strategies for the player.

\begin{defn}  The respective successor stage survival conditions for the solution concepts $\text{IU}$ and $\text{MIU}$, for iteratively undominated and iteratively undominated in the memory-less sense respectively, are given by

\begin{equation} \tag{\text{IU}} 
X^{\alpha+1}_i = \{ s \in X^{\alpha}_i \ | \ \forall \mu \in \Delta(X_i^{0}) \ \exists t \in X_{j}^{\alpha} \ \pi_i(s,t) \geq \pi_i(\mu,t) \}, \mbox{ and}
\end{equation}
\begin{equation} \tag{\text{MIU}}
X^{\alpha+1}_i = \{ s \in X^{\alpha}_i \ | \ \forall \mu \in \Delta(X_i^{\alpha}) \ \exists t \in X_{j}^{\alpha} \ \pi_i(s,t) \geq \pi_i(\mu,t) \}.
\end{equation}

\noindent We again start from $X_i^{0} = X_i$ for each $i$, with intersections taken at limit stages as in Definition 2.4.  \end{defn}

\begin{defn} For each player $i$ there are Polish spaces $T_i$, called the universal type space, and  homeomorphisms $$\lambda_i:T_i \rightarrow \Delta(X_j \times T_j).$$  
The construction of $T_i$, based on \cite{brand2}, is given in Definition 4.1 below (see also \cite{dekel} or \cite{brand} for additional background).  We set $$\text{RCBR}^1_i(G) = \{(s,t) \in X_i \times T_i \ | \ s \mbox{ is a best response to } \mbox{marg}_{X_j}\lambda_i(t) \},$$ and for $n>0$, $$\text{RCBR}_i^{n+1}(G) = \{(s,t) \in X_i \times T_i \ | \ \lambda_i(t)(\text{R}_j^{n}(G)) = 1\}.$$ 
We call $$\text{RCBR}_i(G) = \{ s \in X_i \ | \ \exists t \in T_i \ \forall n \ (s,t) \in \text{RCBR}_i^n(G)\}$$ 
the set of strategies for player $i$ which are consistent with Rationality and Common Belief in Rationality.  \end{defn}

\begin{lem} (Lubin)  If $X,Y$ are Polish spaces, $A \subseteq X \times Y$ is analytic and $\mu \in \Delta(X)$ with $$\mu(\mbox{proj}_{X}(A)) = 1,$$ then there is $\nu \in \Delta(X \times Y)$ with $\nu(A) = 1$ and $\mu = \mbox{marg}_X(\mu)$. \end{lem}

\begin{proof}  This is proved in \cite{lubin} and stated in this form as Lemma 6 of \cite{arieli}.  By the Jankov-von Neumann Uniformization Theorem (see 18.1 and 29.9 of \cite{kechris}) there is $$f: \mbox{proj}_{X}(A) \rightarrow Y$$ which is analytically measurable (and hence universally measurable).  We define $$\nu(S) = \{ x \ | \ (x,f(x)) \in S \}$$ for Borel $S \subseteq X \times Y$.  It can be easily shown that $\nu$ is the desired measure.
\end{proof}

\begin{thm} (Fundamental Theorem of EGT) Suppose $G$ is a game with compact Polish strategy spaces with continuous payoffs.  Then  $$\text{RAT}(G) = \text{MRAT}(G) = \text{IU}(G) = \text{MIU}(G) = \text{RCBR}(G).$$
The iterated deletion procedures converge at the end of the finite stages.
\end{thm}

\begin{proof}  To see $$\text{RAT}(G) = \text{MRAT}(G),$$ assume otherwise and let $\beta$ be least with 
some 
$$s \in \text{MRAT}_i^{\beta+1}(G) \setminus \text{RAT}^{\beta+1}(G).$$  Let $t \in X_i$ be a best response to a measure $\mu \in \Delta(\text{RAT}^{\beta}_j)$ that witnesses $s \in \text{MRAT}_i^{\beta+1}(G)$.  Then $\mu$ witnesses that $t$ survives until stage $\beta$ hence $t \in \text{RAT}_i^{\beta}(G)$, a contradiction.  Now, by induction, using Proposition 2.3, and noting that each $\text{RAT}^{\beta}_i(G)$ is closed and hence compact, we have $$\text{IU}^{\beta}_i(G) = \text{RAT}^{\beta}_i(G)$$ and 
$$\text{MIU}^{\beta}_i(G) = \text{MRAT}^{\beta}_i(G).$$ 

Thus all four procedures coincide.  To see that $\text{RAT}(G)$, and hence the other procedures, converges at the end of the finite stages, suppose $s \in \text{RAT}^{\omega}_i(G)$ and for each $n$, $s$ is a best response to some $\mu_n \in \Delta(X_j)$ with $\mu_n(\text{RAT}^n_j(G)) = 1$.  By compactness, some sub-sequence converges to some $\mu \in \Delta(X_j)$.  Thus $\mu(\text{RAT}^n_j(G)) = 1$ for each $n$ as these sets are closed so $\mu( \text{RAT}^{\omega}_j(G)) = 1$ and it follows from continuity that $s$ is a best response to $\mu$.

We now note, as mentioned in the introduction, for each $i$ and $s \in \text{RCBR}_i(G)$ there is a witness $t$ such that $s$ is a best response to $\mbox{marg}_{X_j}\lambda_i(t)$ and $$\mbox{marg}_{X_j}\lambda_i(t)(\text{RCBR}_j(G)) = 1$$ from which it follows by induction that $\text{RCBR}_i(G) \subseteq \text{RAT}_i(G)$ (see also Lemma 3.5 below).  Thus, it remains to show $\text{RAT}_i(G) \subseteq \text{RCBR}_i(G)$.  We note that each $\text{RAT}_i(G)$ is nonempty and compact as an intersection of nonempty compact sets (this uses the Portmanteau Theorem 17.20 of \cite{kechris}).  

\noindent Let $$\bar{\text{R}}_{i}^{1} = \{ (s,t) \in \text{RCBR}^1_i \ | \ s \in \text{RAT}_i \mbox{ and }\mbox{marg}_{X_j} (\lambda_i(t))( \text{RAT}_j) = 1 \}$$ and for $n>0$ let $$\bar{\text{R}}^{n+1}_i = \{(s,t) \in \bar{R}^{n}_i \ | \ \lambda_i(t)(\bar{\text{R}}^{n}_j) = 1\}.$$  A straightforward induction, again  using the Portmanteau Theorem and the continuity of the marginal function $$\mbox{marg}_{X_j} : \Delta(X_j \times T_j) \rightarrow \Delta(X_j),$$ shows that $\bar{\text{R}}^{n}_i$ is closed for each $i$ and $n$.  It is immediate that each $\bar{\text{R}}^{n}_i \subseteq \text{RCBR}_i^{n}$.  We will show that $$\text{RAT}_i(G) = \mbox{proj}_{X_i}(\bar{\text{R}}^{n}_i)$$ by induction.

If $s \in \text{RAT}_i$ there is $\mu \in \Delta(X_j)$ such that $\mu(\text{RAT}_j(G)) = 1 $ such that $s$ is a best response to $\mu$.  We can apply Lemma 2.11 to lift $\mu$ to $\nu \in \Delta(X_j \times T_j)$ with $\mbox{marg}_{X_j}(\nu) = \mu$.  Then, as $\lambda_i$ is surjective, we find $t$ with $\lambda_i(t) = \nu$ so that 
$(s,t) \in \bar{\text{R}}^{1}_i$.  The induction step follows similarly with $\bar{\text{R}}_j^n$ replacing $X_j \times T_j$ as the lifting target. That is, given $s \in \text{RAT}_i$ and $\mu$ as above, we lift to $\nu \in \Delta(X_j \times T_j)$ with $\mbox{marg}_{X_j}(\nu) = \mu$ and $\nu(\bar{\text{R}}_j^n)=1$.  Again because $\lambda_i$ is surjective, we find $t$ with $\lambda_i(t) = \nu$ and, because the $\bar{\text{R}}^{n}_j$ are decreasing in $n$, we conclude that  $(s,t) \in \bar{\text{R}}^{n+1}_i$.
Thus for each $s \in \text{RAT}_i(G)$ the sets $\{ t \ | \ (s,t) \in \bar{\text{R}}^{n}_i \}$ are decreasing, compact and nonempty, and so have nonempty intersection.  This establishes the result. \end{proof}

\begin{rem}  We used Lemma 2.11 in the proof of the Fundamental Theorem but 
there are easier uniformization results in the compact case.  For example, we can view $\bar{\text{R}}_j^{n}$ as a closed-valued correspondence from $\text{RAT}_j(G)$ to $T_j$ which is easily seen to be upper hemi-continuous (it has a closed graph; see 17.11 of \cite{ali}) and apply the Kuratowski and Ryll-Nardzewski Selection Theorem (see Thm. 5.2.1 of \cite{sri} or 18.13 of \cite{ali}) to produce a Borel map 
$$\phi_j^n:\text{RAT}_j(G) \rightarrow T_j$$ with $$(\bar{s},\phi_j^n(\bar{s})) \in \bar{\text{R}}_j^n$$ for every $\bar{s} \in \text{RAT}_j(G)$.  We can then define the Borel measure $\nu \in \Delta(X_j \times T_j)$ by $$\nu(S) = \mu( \{ \bar{s} \ \in X_j \ | \ (\bar{s},\phi_j^n(\bar{s})) \in S \} ).$$
Thus $\nu(\bar{\text{R}}_j^n) = 1$ and $\mbox{marg}_{X_j}(\nu) = \mu$ so there is $t$ with $\lambda_i(t) = \nu$ and $(s,t) \in \bar{\text{R}}_i^{n+1}$ as desired.\footnote{See also 12.13, 18.18, and 35.46 of \cite{kechris}.  There is a Borel map $c$ defined on the space of compact subsets of $T_j$ which selects a member of each such set.  The sets $\{t \in T_j \ | \ (s,t) \in \bar{\text{R}}_j^n \}$ are compact and nonempty for $s \in \text{RAT}_j(G)$ so
 $$\phi_j^n(s) = c( \{t \in T_j \ | \ (s,t) \in \bar{\text{R}}_j^n \} )$$ 
defines a Borel uniformization of $\bar{\text{R}}_j^n$ on $\text{RAT}_j(G)$.} \end{rem}

\section{Interactive Epistemics}
Assume $X_1$ and $X_2$ are Polish spaces and for each $i$ let $\Delta(X_i)$ denote the Polish space of Borel probability measures on $X_i$.  As before we think of $X_i$ as the strategy or choice space for player $i$, and in a context where $i$ is used, we sometimes use $j$ to represent the other player.  Let $E_i \subseteq X_i \times \Delta(X_j)$ for each $i$, and let $E$ denote the pair $(E_1,E_2)$.  We think of $E_i$ as a relation between strategies and (first-order) beliefs for player $i$. This will be our basic object of study.

\begin{defn}  A strategy rectangle $B_1 \times B_2 \subseteq X_1 \times X_2$ is called $\text{E}$-justified if for every $s \in B_i$ there is $\mu \in \Delta(X_j)$ with $(s,\mu) \in \text{E}_i$ such that $\mu(B_j) = 1$.
Thus, setting $$\text{RAT}(\text{E})_i = \bigcup \{ \ \mbox{proj}_{X_i}(B) \ | \ B \mbox{ is an } \text{E}\mbox{-justified rectangle} \ \}$$
we see that $${\text{RAT}(\text{E})} = \text{RAT}(\text{E})_1 \times \text{RAT}(\text{E})_2$$ is itself an $\text{E}$-justified rectangle and hence the maximal one.  We say that a strategy $s \in \text{RAT}(\text{E})_i$ is $\text{E}$-rationalizable for player $i$ (cf. Remark 2.6). 
\end{defn}

\begin{rem} The sets $\text{RAT}(\text{E})_i$ can also be characterized by iteratively deleting non-$\text{E}$-justified strategies as follows. We 
set $\text{RAT}(\text{E})_i^{0}= X_i$, take intersections at limit stages, and set 
$$\text{RAT}(\text{E})_i^{\alpha+1} = \{ s \in \text{RAT}(\text{E})_i^{\alpha} \ | \ \mu(\text{RAT}(\text{E})^{\alpha}_j) = 1 \mbox{ for some } \mu \mbox{ with } (s,\mu) \in \text{E}_i \}.$$ We set $\alpha(E)$ equal to the least $\alpha$ such that $\text{RAT}(\text{E})^{\alpha}_i = \text{RAT}(\text{E})_i^{\alpha+1}$ for both $i$.  It is easy to see that $$\text{RAT}(\text{E})^{\alpha(E)}_1 \times \text{RAT}(\text{E})^{\alpha(E)}_2  = \text{RAT}(\text{E})_1 \times \text{RAT}(\text{E})_2.$$  Indeed, the strategies in $\text{RAT}(\text{E})_1 \times \text{RAT}(\text{E})_2$ survive iterated deletion and $$\text{RAT}(\text{E})^{\alpha(E)}_1 \times \text{RAT}(\text{E})^{\alpha(E)}_2$$ is itself $\text{E}$-justified by definition.  \end{rem}

\begin{defn}   For each player $i$ there are Polish spaces $T_i$, called the universal type space for player $i$, and 
homeomorphisms  (we use the construction in \cite{brand2}; see Definition 4.1 below) $$\lambda_i:T_i \rightarrow \Delta(X_j \times T_j).$$    The relations $\text{E}_i$ give rise to sets $$\text{RCBR}(\text{E})^1_i = \{(s,t) \in X_i \times T_i \ | \ (s,\mbox{marg}_{X_j}\lambda_i(t)) \in \text{E}_i \},$$ and for $n >1$ 

$$\text{RCBR}(\text{E})_i^{n+1} = \{(s,t) \in \text{RCBR}(\text{E})_i^{n} \ | \ \lambda_i(t)(\text{RCBR}(\text{E})_j^{n}) = 1\}.$$  We set 
$$\text{RCBR}(\text{E})_i  = \{ s \in X_i \ | \ \exists t \in T_i \ (s,t) \in \bigcap_{n>0}  \text{RCBR}(\text{E})_i^n \}$$ and 
$$\text{RCBR}(\text{E}) = \text{RCBR}(\text{E})_1 \times \text{RCBR}(\text{E})_2.$$ We say that $\text{RCBR}(\text{E})_i$ are the strategies for $i$ which are consistent with "$\text{E}$ and common belief in $\text{E}$". 
\end{defn}

\begin{rem}  Recall that a subset $A \subseteq X$ of a Polish space is universally measurable if $A$ is measurable by the completion of $\mu$ for every $\mu \in \Delta(X)$ (see 17A of \cite{kechris}).  In Definition 3.3, if we assume that each $\text{E}_i$ is universally measurable, then $\text{RCBR}(\text{E})^1_i$ is universally measurable as a continuous pre-image of a universally measurable set and the universal measurability of $\text{RCBR}(\text{E})_i^{n+1}$ follows from Cor. 7.46 of \cite{shreve}. \end{rem}

\begin{lem}  $\text{RCBR}(\text{E})$ is $E$-justified and hence $\text{RCBR}(\text{E}) \subseteq \text{RAT}(\text{E}).$  \end{lem}
\begin{proof} Let $s$ be a strategy for player $i$ which has a type $t \in T_i$ at which $\text{E}$ and common belief in $\text{E}$ holds.   
We have $$\lambda_i(t)(\bigcap_{n>0} \text{RCBR}(\text{E})^n_j) = 1$$ so $\mu(\text{RCBR}(\text{E})_j) = 1$ where $\mu = \mbox{marg}_{X_j}(\lambda_i(t))$. To see this let $$B \subseteq \bigcap_{n>0} \text{RCBR}(\text{E})^n_j$$ be Borel with $\lambda_i(t)(B) = 1$.  Then $$\mbox{proj}_{X_j}(B) \subseteq \text{RCBR}(\text{E})_j$$ is analytic and 
$\mu(\mbox{proj}_{X_i}(B)) =1$ so $\mu(\text{RCBR}(\text{E})_j) = 1$.
We have $(s,\mu) \in \text{E}_i$ by assumption.  \end{proof}

\begin{thm}  Suppose each $X_i$ is compact and each $E_i \subseteq X_i \times \Delta(X_j)$ is closed.  Then $$\text{RAT}(\text{E}) = \text{RCBR}(\text{E}).$$
\end{thm}
\begin{proof}  By Lemma 3.5 we only need to establish the reverse inclusion.  Again, we can assume that each $\text{RAT}(\text{E})_i$ is nonempty, and is therefore  compact as an intersection of nonempty compact sets.  Let $$S_{i}^{1} = \{ (s,t) \in \text{RCBR}(\text{E})_i^1 \ | \ s \in \text{RAT}(\text{E})_i \mbox{ and } \mbox{marg}_{X_j} (\lambda_i(t))( \text{RAT}(\text{E})_j) = 1 \}$$ and for $n>0$ let $$S^{n+1}_i = \{(s,t) \in S^{n}_i \ | \ \lambda_i(t)(S^{n}_j) = 1\}.$$  A straightforward induction using the assumption on $\text{E}$ and the Portmanteau Theorem (17.20 of \cite{kechris}) shows that $S^{n}_i$ is closed for each $i$ and $n$.  We show that $$\text{RAT}(\text{E})_i = \mbox{proj}_{X_i}(S^{n}_i)$$ by induction.  Since $S_i^{n} \subseteq \text{RCBR}(\text{E})_i^{n}$ the result follows as in Theorem 2.11.  For the base case, if $s \in \text{RAT}(\text{E})_i$ there is $\mu \in \Delta(X_j)$ such that $(s,\mu) \in \text{E}_i$ and $\mu(\text{RAT}(\text{E})_j) = 1$.  Here we can just lift to $\nu \in \Delta(X_j \times T_j)$ defined by $\nu(S \times T_j) = \mu(S)$ so that $\mbox{marg}_{X_j}(\nu) = \mu$.  There must be $t \in T_i$ with $\lambda_i(t) = \nu$ so $(s,t) \in S^{1}_i$.  For the induction step, we can lift the same $\mu$ to $\nu$ on $X_j \times T_j$ by Lemma 2.11 with $\nu(\text{RCBR}(\text{E})^{n}_j) = 1$
and again find $t \in T_i$ with $\lambda_i(t) = \nu$ so that $\lambda_i(t) (\text{RCBR}(\text{E})^{n}_j) = 1$. 
Thus $(s,t) \in S^{n+1}_i$, as the $\text{RCBR}(\text{E})^{n}_j$ are decreasing.

As in Theorem 2.11, for $s \in \text{RAT}(\text{E})_i$, the set of $t$ with $(s,t) \in S_i^n$ for every $n$ is an intersection of decreasing sequence of nonempty compact sets, and thus there is $t \in T_i$ with $$(s,t) \in \bigcap_{n<\omega} S_i^{n} \subseteq \text{RCBR}(\text{E})^n_i$$ so $s \in \text{RCBR}(\text{E})_i$.   \end{proof}

\begin{cor}  Rationalizability is equivalent to Rationality and Common Belief in Rationality for continuous games played on compact Polish spaces.
\end{cor}

\begin{proof}  Suppose each $\pi_i:X_1 \times X_2 \rightarrow \mathbb{R}$ is a bounded payoff function so that $G = (\pi_1,\pi_2)$ is a strategic form game.   Let 
$$\text{E}_1 = \text{E}^G_1 = \{ (s, \mu) \in X_1 \times \Delta(X_2) \ | \ \forall t \in X_1 \ \pi_1(s,\mu) \geq \pi_1(t,\mu) \}$$
$$\text{E}_2 = \text{E}^G_2 = \{ (s, \mu) \in X_2 \times \Delta(X_1) \ | \ \forall t \in X_2 \ \pi_2(\mu,s) \geq \pi_2(\mu,t) \}$$  denote the best response relation for each player, and set $\text{E} = (\text{E}_1,\text{E}_2)$  Thus $$\text{RAT}(\text{E}) = \text{RAT(G)} \mbox{ and } \text{RCBR}(\text{E}) = \text{RCBR(G)}.$$

\noindent Observe that the relations $\text{E}_i$ as above are closed if each $\pi_i$ is continuous, and so the result follows from Theorem 3.6.  \end{proof}

\section{Belief Hierarchies and Determinacy I}

\noindent We continue to use the notation of Section 3.  Given Lemma 3.5, we are interested in the inclusion 
$$\text{RAT}(\text{E})_1 \times \text{RAT}(\text{E})_2 \subseteq \mathbf{\text{RCBR}(\text{E})}_1 \times \mathbf{\text{RCBR}(\text{E})}_2,$$ and assumptions under which it holds for a pair of universally measurable relations $\text{E}= (\text{E}_1,\text{E}_2)$.  In this section, we show that the inclusion follows from the Axiom of Real Determinacy.  We will now need the construction of the universal type space from \cite{brand2}.

\begin{defn}
The spaces $T_i$ and maps $\lambda_i$ referenced in the definition of $\text{RCBR}$ are called the universal type structure in the literature.    A type is a certain hereditarily coherent sequence of measures.   Adapting \cite{brand2}, for each i set $X^0_i = X_j$ and for $n>0$ define 
$$X^n_i = X^{n-1}_i \times \Delta(X_j^{n-1}).$$
Thus $$X^n_i = X_j \times \Delta(X_j^0) \times \Delta(X_j^1) \times ... \times \Delta(X_j^{n-1}).$$  We think of $X_i^n$ as the nth order space of uncertainty for player i, and $\Delta(X_i^n)$ represents the space of $n^{th}$-order beliefs for player $i$.  Thus $\delta \in \Delta(X_i^n)$ is a measure for player $i$ over the strategies for player $j$ and the beliefs of player $j$ at every level up to and including level $n-1$.  For each i, let $$T^0_i = \prod_{n>0} \Delta(X^{n-1}_i)$$
\noindent which is Polish in the product topology.  $T_i \subset T_i^0$ will be a certain closed subspace.
Let $T^{1}_i$ consist of all sequences $$\delta_i = (\delta_i^1,\delta_i^2,...) \in T^0_i,$$ that is each $\delta_i^n \in \Delta(X_i^{n-1})$ for $n>0$, which are coherent in the sense that $$\mbox{marg}_{X^{n-2}_i}(\delta^n_i) = \delta^{n-1}_i$$ for every $n>1$.
$T^{1}_i$ is a closed subspace of $T^{0}_i$ and hence Polish.  For each $\delta_i \in T^{1}_i$ there is $\delta^{*}_i \in \Delta(X_j \times T^0_j)$ with $\delta^{n+1}_i = \mbox{marg}_{X_i^n} \delta^{*}_i$ for every $n$ by the Kolmogorov Extension Theorem.  This defines a homeomorphism  $$\lambda_i :T^{1}_i \rightarrow \Delta(X_j \times T_j^{0})$$
\noindent given by $\lambda_i(\delta_i) = \delta^{*}_i$.  For $k>1$ define $$T^{k}_i = \{ \delta_i \in T^{1}_{i} \ | \ \delta_i^{*}(X_j \times T_j^{k-1})=1 \}$$
and let $$T_i = \bigcap_{k<\omega} T_i^k,$$
\noindent which is closed and hence Polish.  For $\delta_i \in T_i$ we have 
$\delta_i^{*}(X_j \times T_j) =1$ so 
it follows that $$\lambda_i :T_i \rightarrow \Delta(X_j \times T_j)$$ 
\noindent is a homeomorphism (see Prop. 1 and Prop. 2 of \cite{brand2}).  The elements of $T_i$ are called types for player $i$ and are said to satisfy common certainty of coherency.  \end{defn}

\begin{rem} (Cylinder Sets) Suppose $\delta_i \in T_i$ is a type.  Suppose $B$ is a Borel subset of  $$X_i^n = X_j \times \Delta(X_j^0) \times \Delta(X_j^1) \times ... \times \Delta(X_j^{n-1}).$$  Then 
$$\lambda_i(\delta_i)( \{ (t,\delta_j) \in X_j \times T_j \ | \ (t,\delta_j^1,...,\delta_j^n) \in B \} ) = \delta_i^{n+1}(B).$$  This is by definition of the Kolmogorov Extension in our setting.  The measure $\lambda_i(\delta_i)$ is uniquely determined by its values on the cylinder sets.

\end{rem}

\begin{rem}  The version of the Kolmogorov Extension Theorem we use here only makes use of the Axiom of Dependent Choice.  \end{rem}

\begin{defn} (Hereditary Coherence) We say a measure, specifically an $n^{th}$ order belief, is hereditarily coherent if it concentrates on coherent sequences of hereditarily coherent measures.  Similarly, we say a type is hereditarily coherent if it is a coherent sequence of hereditarily coherent measures.  More precisely, we say $(s,\delta_j^1,...,\delta_j^n) \in X_i^n$ is coherent if $$\mbox{marg}_{X^{k-2}_j}(\delta^k_j) = \delta^{k-1}_j$$ for $k=2,...,n$.  We say $\delta_i^{n+1} \in \Delta(X_i^n)$ concentrates on coherent sequences if the set of coherent sequences $(s,\delta_j^1,...,\delta_j^{n}) \in X_i^n$ as above is assigned measure one by $\delta_i^{n}$.  Now we inductively define $\delta^{n+1}_i$ to be hereditarily coherent, for $n \geq 2$, as follows.  For $n=2$ we require only that $\delta^{n+1}_i$ concentrate on coherent sequences.   For $n > 2$ we require that it concentrates on coherent sequences of hereditarily coherent measures.  
\end{defn}

\begin{lem}  (Local Version of Common Certainty of Coherency) Suppose $$\delta_i = (\delta_i^1,\delta_i^2,...) \in T_i^1$$ is a coherent sequence of measures.  Then $\delta_i \in T_i$, that is $\delta_i$ satisfies common certainty of coherency, if and only if each $\delta_i^n$ is a hereditarily coherent measure.  \end{lem}
\begin{proof} Let $M_1$ be the set of $n^{th}$-order beliefs, for any $n$, which concentrate on coherent sequences.  Let $M_{k+1}$ denote the set of $n^{th}$-order beliefs for any $n$ which types concentrate on coherent sequences of measures in $M_k$.  Then it is easy to see by induction, using Remark 4.2, that $T_i^{k+1}$, from Definition 4.1, is equal to the set of coherent sequences of measures in $M_{k}$.  Moreover, it is easy to see that the intersection $T_i$ is precisely the set of coherent sequences of hereditarily coherent measures. \end{proof}

\begin{defn} (Local Version of RCBR) We need a local reformulation of Definition 3.3, that is the notion of a type $\delta_i$ witnessing $s \in \text{RCBR}(\text{E})_i$.  We start by defining 
$$\text{RCBR}(\text{E})_i^{*,n} \subseteq X_i \times T_i \upharpoonright (n+1) \subset X_i \times (\Delta(X_i^0) \times \Delta(X_i^1) \times ... \times \Delta(X_i^n))$$ for $n\geq 1$.  For $n=1$ we define $\text{RCBR}(\text{E})_i^{*,1} \subseteq X_i \times \Delta(X_j)$ by $$\text{RCBR}(\text{E})_i^{*,1} = \text{E}_i.$$  Similarly we can define 
$$\text{RCBR}(\text{E})_i^{*,2} \subseteq X_i \times \Delta(X_j) \times \Delta(X_j \times \Delta(X_i) ) $$ by $\text{RCBR}(\text{E})_i^{*,2} =$ $$ \{ (s,\mu,\nu) \ | \ (s,\mu) \in \ \text{RCBR}(\text{E})_i^{*,1} \wedge \mu = \mbox{marg}_{X_j}(\nu) \wedge \nu (\text{RCBR}(\text{E})_j^{*,1}) = 1 \} = $$ $$ \{ (s,\mu,\nu) \ | \ (s,\mu) \in \ \text{E}_i \wedge \mu = \mbox{marg}_{X_j}(\nu) \wedge \nu (\text{E}_j) = 1 \}    $$
\noindent For the general case, we define $\text{RCBR}(\text{E})_i^{*,n}$ to be the set of hereditarily coherent (see Definition 4.4) sequences $(s,\delta_i^1,...,\delta_i^n)$ with $$(s,\delta_i^1,...,\delta_i^{n-1}) \in  \text{RCBR}(\text{E})_i^{*,n-1}$$ satisfying $$\delta_i^{n} (\text{RCBR}(\text{E})_j^{*,n-1}) = 1 .$$ \end{defn}

\begin{rem}  Note that if $(s,\delta_i^1,...,\delta_i^{n-1})$ is hereditarily coherent, and $\delta_i^{n} \in \Delta(X_i^{n-1})$ concentrates on hereditarily coherent measures and coheres in that $$\mbox{marg}_{X_i^{n-2}}(\delta_i^{n}) =  \delta_i^{n-1},$$ then 
$(s,\delta_i^1,...,\delta_i^{n})$ is hereditarily coherent.  
\end{rem}

\begin{lem}  Suppose $s \in X_i$ and $\delta_i = (\delta_i^1, \delta_i^2, \delta_i^3,...) \in T_i$.  Then for every $n>0$, $$(s,\delta_i) \in \text{RCBR}(\text{E})_i^n \mbox{  if and only if  } (s,\delta_i^1,...,\delta_i^n) \in \text{RCBR}(\text{E})_i^{*,n}.$$

\noindent Moreover, if $s \in X_i$ and $(s,\mu_i^1,...,\mu_i^n) \in  \text{RCBR}(\text{E})_i^{*,n}$ then there is $\delta_i$ with $$(s,\delta_i) \in \text{RCBR}(\text{E})_i^n$$ and $(\mu_i^1,...,\mu_i^n) = (\delta_i^1,...,\delta_i^n).$

\end{lem}
\begin{proof}  By induction.  For the case $n = 1$ we have $\mbox{marg} (\lambda_i(\delta_i)) = \delta_i^1$ 
which suffices.  For $n=2$, the condition becomes  $\lambda_i(\delta_i)( \text{RCBR}(\text{E})_i^1) = 1$ which is equivalent to $$\delta_i^2(\text{RCBR}(\text{E})_j^{*,1}) ) = 1.$$  And so on.  For the second part, suppose $(s,\mu) \in \text{RCBR}(\text{E})_i^{*,1}$.  We can lift $\mu$ to $\hat{\mu} \in \Delta(X_j \times T_j)$ by Lemma 2.11 which must be $\lambda_i(\delta_i)$ for some $\delta_i \in T_i$.  We will have $\mu = \delta_i^1$ so $(s,\delta_i) \in \text{RCBR}(\text{E})_i^1$.  The general case is the same.
\end{proof}

\begin{thm}  Assume $\mathsf{AD}_{\mathbb{R}}$.  Suppose that $\text{E}$ is a pair of relations between strategies and beliefs on underlying Polish spaces.  Then $$ \text{RAT}(\text{E}) = \text{RCBR}(\text{E}).$$  \end{thm}
\begin{proof} We will use two consequences of $\mathsf{AD}_{\mathbb{R}}$, uniformizations and measurability.

\begin{enumerate}
    \item If $X$ is Polish and $\mu \in \Delta(X)$ then every subset of $X$ is $\mu$-measurable
    \item If $X,Y$ are Polish and $A \subseteq X \times Y$ then there is $B \subseteq A$ such that for every $x \in X$ if there is $y$ with $(x,y) \in A$ then there is a unique $y$ with $(x,y) \in B$
\end{enumerate}

\noindent That all sets are Lebesgue measurable follows from $\mathsf{AD}$ (see 27.9 of \cite{kanamori}), and hence (1) follows by the argument of 17.41 and 22.10 of \cite{kechris}, while uniformization (see \cite{larson} and Theorem 6.8 of \cite{larson2}) is actually equivalent to $\mathsf{AD}_{\mathbb{R}}$ in the presence of $\mathsf{AD}$.  By Lemma 3.5, we may assume that each $\text{RAT}(\text{E})_i$ is nonempty.  By (2) there is a uniformization map $$\mu^i: \text{RAT}(\text{E})_i \rightarrow \Delta(X_j)$$ 
with $(s,\mu^i(s)) \in \text{E}_i$ and $$\mu^i(s) (\text{RAT}(\text{E})_j) = 1$$ for every $s \in \text{RAT}(\text{E})_i$.  We will show that this map lifts to a map $$t^i: \text{RAT}(\text{E})_i \rightarrow T_i$$ with the property $$\lambda_i(t^i(s))(S) = \mu^{i}(s)(\mbox{proj}_{X_j}(S \cap t^j))$$ for any Borel set $S \subseteq X_j \times T_j$.  Here we are identifying $t^j \subset X_j \times T_j$ with its graph.  So in particular $$\mu^i(s) = \mbox{marg}_{X_j}(\lambda_i(t^i(s))).$$  Given this, it follows that $t^i(s)$ witnesses $s \in \text{RCBR}(\text{E})_i$.  To see this, we show by induction on $n$ that $$t^{i} \subseteq \text{RCBR}(\text{E})^n_i$$ for each $i$, again identifying $t^i$ with its graph.  The base case $n=1$ follows as $$(s,\mbox{marg}_{X_j}(\lambda_i(t_i(s))) = (s,\mu^i(s)) \in \text{E}_i.$$  For $n>1$ we have $(s,t^i(s)) \in \text{RCBR}(\text{E})_i^{n}$ by induction hypothesis, and since $\lambda_i(t^i(s)) (t^j) = 1$ we have $$\lambda_i(t^i(s))(\text{RCBR}(\text{E})_j^{n}) = 1$$ also by induction hypothesis, hence $(s,t^i(s)) \in \text{RCBR}(\text{E})_i^{n+1}$.  

It remains to define the lift maps $t^i$.   For each $j$ and $s \in \text{RAT}(\text{E})_j$ define $\delta_j^1(s) = \mu^j(s)$ and by induction on $k$, for each $j$, define for a Borel set $S$,
$$\delta_j^{k+1}(s)(S) = \mu^j(s)( \{ t \ | \    (t,\delta_i^1(t),...,\delta_i^{k}(t)) \in S \}.$$  By our measurability hypothesis, $\delta_j^{k+1}(s) \in \Delta(X_j^{k})$ is a Borel measure.  Set $$t^j(s) = (\delta_j^{1}(s),\delta_j^{2}(s),\delta_j^{3}(s),..).$$  
It is straightforward to check that $$\mbox{marg}_{X_j^{k-1}}(\delta_j^{k+1}) = \delta_j^{k}$$
for each $j$ and $k$ so by induction and Lemma 4.5, $t^j(s)$ satisfies common certainty of coherence, and we conclude that $t^j(s) \in T_j$.

We want to show that $$\lambda_i(t^i(s))(S) = \mu^{i}(s)(\mbox{proj}_{X_j}(S \cap t^j))$$ for any Borel set 
$S \subseteq X_j \times T_j$.  Note that $\lambda_i$ is obtained by the Kolmogorov Extension Theorem (see 17.16 of \cite{kechris} and Remark 4.2) and so it is immediate that $\lambda_i(t^i(s))(t_j) = 1$ as $t^j$ is a countable intersection of cylinder sets assigned measure one by the marginals.  That is $$\lambda_i(t^i(s))( \hat{C}_n ) = \delta^{n+1}_i(s)(C_n) = \mu^i(s)(\text{RAT}(\text{E})_j) = 1$$ where $$C_{n} = \{ (t, \delta_j^1(t),...,\delta_j^n(t)) \ | \ t \in \text{RAT}(\text{E})_j \}$$ and $$\hat{C}_n = \{ (t,\delta_j) \in X_j \times T_j \ | \  \forall k \leq n \ \delta_j^k = \delta_j^k(t) \} .$$  Thus $$\lambda_i(t^i(s))( \bigcap_{n<\omega} \hat{C}_n ) = \lambda_i(t^i(s))(t^j) = 1.$$

\noindent It similarly follows that $$\lambda_i(t^i(s))(S) = \lambda_i(t^i(s))(S \cap t^j)) = \mu^{i}(s)(\mbox{proj}_{X_j}(S \cap t^j))$$ for cylinder sets and by extension for Borel sets.
\end{proof}

\begin{cor}   $\mathsf{AD}_{\mathbb{R}}$ implies that Rationalizability is equivalent to Rationality and Common Belief in Rationality for any strategic form game with bounded payoffs played on Polish strategy spaces.  \end{cor}
\begin{proof}  Just like Corollary 3.7.  \end{proof}

\begin{rem}  The proofs of the Fundamental Theorem from Theorem 
2.12 and 3.6 used only the existence of a type structure.  We now reprove Theorem 3.6 using belief hierarchies.  Let 
$$C^i:\mathcal{K}(\Delta(X_i)): \rightarrow \Delta(X_i)$$ be a Borel map (see 12.13 of \cite{kechris}) on the Polish space of compact subsets of $\Delta(X_i)$ that selects an element from any non-empty compact set of measures, noting that each $\Delta(X)_i)$ is compact.  By the Portmanteau Theorem and assumption on $\text{E}_i$, $\text{RAT}(\text{E})_i$ is compact and so the map 
$$\mu^i(s) = C^j( \{ \mu \in \Delta(X_j) \ | \ (s,\mu) \in \text{E}_i \wedge \mu( \text{RAT}(\text{E})_j) = 1 \} )$$ for $s \in \text{RAT}(\text{E})_i$ is Borel Measurable.  Define $\delta_i^1(s) = \mu^{i}(s)$ for each such $s$ and each $i$.  By induction we define 

$$\delta^{n+1}_i(s) = (\text{id}_{X_j},\delta_j^1,...,\delta_j^n)_{*} (\delta_i^1(s))$$
meaning that $\delta^{n+1}(s)$ is the pushforward of the measure $\delta_i^1(s)$ by the map sending $t \in \text{RAT}(\text{E})_j$ to 
$$(\text{id}_{X_j}(t),\delta_j^1(t),...,\delta_j^n(t)).$$
We then define $t^i(s) \in T_i$ by $$t^i(s) = (\delta_i^1(s),\delta^2_i(s),...)$$
for $s \in \text{RAT}(\text{E})_i$, and observe that $$\lambda_i(t^i(s))( \{ (t,\delta_j^1(t),\delta_j^2(t),...) \ | \ t \in \text{RAT}(\text{E})_j \}) = \lambda_i(t^i(s))( \text{graph}(t^j) ) = 1,$$ as in Theorem 4.9., and that $\mbox{marg}_{X_j}(\lambda_i(t^i(s))) = \delta_i^1(s) = \mu^i(s)$.  It then follows that $\text{graph}(t^i) \subseteq \text{RCBR}(\text{E})_i$ as desired.  \end{rem}

\section{Belief Hierarchies and Determinacy II}

We continue the analysis of the last section by defining a real Gale-Stewart game which characterizes $\text{E}$-rationalizability.  We will see that $\text{RCBR}(\text{E})$ follows from the existence of a measurable winning strategy which gives another proof of the fundamental theorem under Determinacy.

\begin{defn}  Suppose that $\text{E}$ is a pair of  relations between strategies and beliefs on underlying Polish spaces $X_i$.  Given $s \in X_i$ let $G^{\text{E}}_s$ denote the following two-player game.  For $n < \omega$, player I plays $(\mu_n,b_n)$ and then II plays $s_{n+1}$.  Setting $s_0 = s$, for each $n$ we must have
\begin{enumerate}
    \item  If $n$ is even then $\mu_n \in \Delta(X_j)$, $(s_n,\mu_n) \in \text{E}_i$ and $b_n \subseteq X_j$ is Borel with $\mu_n(b_n) = 1$
    \item  If $n$ is odd then $\mu_n \in \Delta(X_i)$, $(s_n,\mu_n) \in \text{E}_j$ and $b_n \subseteq X_i$ is Borel with $\mu_n(b_n) = 1$ 

    \item $s_{n+1} \in b_n$
\end{enumerate}

\noindent That is, the first player to break these rules loses.  I wins if the play continues for every stage.  Note that if for any $n$, the strategy $s_{n+1}$ is such that there is no $\mu_{n+1}$ with $(s_{n+1},\mu_{n+1})$ in the appropriate relation $E_i$, then II necessarily wins (in the case that $\text{E}$ is the best response relation from a game this would happen if $s_{n+1}$ is not a best response strategy).  Thus, a play of the game where I wins is a infinite sequence $$(s_0,(\mu_0,b_0),s_1,(\mu_1,b_1),s_2,(\mu_2,b_2),s_3,....)$$ according to the rules above.  This game is easily coded as a real game, since Borel sets and Borel measures can be coded by reals in some fixed manner.  \end{defn}

\begin{prop} Suppose that $\text{E}$ is a pair of  relations between strategies and beliefs on underlying Polish spaces $X_i$, and $s \in X_i$. Then the game $G^{E}_s$ is quasi-determined, and I has winning quasi-strategy if and only if $s \in \text{RAT}(\text{E})_i$.  \end{prop}
\begin{proof}  Suppose $s \in \text{RAT}(\text{E})_i$.  We define a winning quasi-strategy $\sigma$ for player $\text{I}$.  By assumption there is $\mu \in \Delta(X_j)$ such that $(s,\mu) \in \text{E}_i$ and $\mu( \text{RAT}(\text{E})_j) = 1$.  For each such $\mu$ there is a Borel set $b$ such that $b \subseteq \text{RAT}(\text{E})_j$ and $\mu(b) = 1$.  Let $\sigma(\emptyset)$ be the collection of such pairs $(\mu,b)$.  For any such pair $(\mu_0,b_0)$ and any $s_1 \in b_0$ we must have $s_1 \in \text{RAT}(\text{E})_j$ so again there is $\mu_1 \in \Delta(X_i)$ with $(s_1,\mu_1) \in \text{E}_j$ and $\mu( \text{RAT}(\text{E})_i) = 1$.  Thus there is Borel set $b_1$ such that $b_1 \subseteq \text{RAT}(\text{E})_i$ and $\mu_1(b_1) = 1$.  We can collect all such pairs $(\mu_1,b_1)$ into $\sigma( ((\mu_0,b_0),s_1))$.  Continuing in this way we produce a winning quasi-strategy (see 3E of \cite{mosch2}) for player $\text{I}$.

Now suppose $s \notin \text{RAT}(\text{E})_i$ and I plays $(\mu_0,b_0)$ with $(s_0,\mu_0) \in \text{E}_i$ and $\mu_0(b_0) =1$.
We must have $$b_0 \setminus \text{RAT}(\text{E})_j \neq \emptyset$$ so there is a least $\gamma = \gamma_0$ with 
$$b_0 \cap  (\text{RAT}(\text{E})^{\gamma}_j \setminus   \text{RAT}(\text{E})^{\gamma+1}_j) \neq \emptyset$$ and we can collect all such elements $s_1$ in this set into a quasi-strategy for $\text{II}$, that is we set $$\tau((\mu_0,b_0)) = b_0 \cap  (\text{RAT}(\text{E})^{\gamma}_j \setminus   \text{RAT}(\text{E})^{\gamma+1}_j).$$ 

Now let $(\mu_1,b_1)$ be a legal play for $\text{I}$ following some $s_1 \in \tau((\mu_0,b_0)).$  We must have $$b_1 \setminus \text{RAT}(\text{E})^{\gamma_0}_i \neq \emptyset$$ where $\gamma_0$ is the ordinal associated to $\mu_0,b_0$ as above.  So there is a least ordinal $\gamma_1 < \gamma_0$ be such that there is $$s_2 \in \text{RAT}(\text{E})^{\gamma_1}_i \setminus \text{RAT}(\text{E})^{\gamma_{1}+1}_i$$ and we can collect all such $s_2$ into the strategy, that is 

$$\tau( (\mu_0,b_0),s_1, (\mu_1,b_1) ) = b_1 \cap (\text{RAT}(\text{E})^{\gamma_1}_i \setminus \text{RAT}(\text{E})^{\gamma_{1}+1}_i).$$

\noindent  Proceeding in this way we must eventually be able to play $s_n$ which has no accompanying measure in the appropriate $\text{E}_i$, and hence we have defined a winning quasi-strategy for player $\text{II}$.
\end{proof}
\begin{rem}  In $\mathsf{ZFC}$, that is assuming the Axiom of Choice, as one normally does, we can produce a winning strategy (as opposed to quasi-strategy) by the argument above.  Hence, in $\mathsf{ZFC}$, a strategy $s$ is Rationalizable if and only if player $I$ wins the Gale-Stewart game $G^{E}_s$ where $\text{E}$ is the pair of best response relations from the game.  The point is that for Prop. 5.2, we work in $\mathsf{ZF+DC}$, that is in set theory with the Axiom of Dependent Choice $\mathsf{DC}$, which is consistent with determinacy axioms.  \end{rem}

\begin{defn}   It is immediate that the set of strategies $s$ for which I has a winning strategy in $G^{E}_s$ is an $\text{E}$-justified set in the sense of Definition 3.1.  Let us denote these strategies by $\text{RAT}(\text{E})^{*}_i$ for player $i$ and set $\text{RAT}(\text{E})^* = \text{RAT}(\text{E})^{*}_1 \times \text{RAT}(\text{E})^{*}_2$.  Thus $$\text{RAT}(\text{E})^{*} \subseteq \text{RAT}(\text{E})$$ without assuming the axiom of choice. \end{defn}

\begin{prop}  Assume $\mathsf{AD}_{\mathbb{R}}$.  Then $\text{RAT}(\text{E})^{*} = \text{RAT}(\text{E})$.  \end{prop}

\begin{proof}   Under $\mathsf{AD}_{\mathbb{R}}$, each game $G^{E}_s$ is determined, and so the equivalence follows from Proposition 5.2 as the player with a winning quasi-strategy must be the player with the winning strategy. \end{proof}

\begin{thm}  Assume all sets are Lebesgue measurable.
Then $$\text{RAT}(\text{E})^{*} \subseteq 
\text{RCBR}(\text{E}).$$   
\end{thm}
\begin{proof}  We show that under the measurability hypothesis, a winning strategy in $G_s^{\text{E}}$ gives rise to a type certifying $s \in \text{RCBR}(\text{E})_i$.  By induction we will define a $n^{th}$-order belief $\delta_{i,\tau}^n$ for every $i$, $s \in \text{RAT}(\text{E})_i^{*}$, and winning  strategy $\tau$ in $G_{s}^{E}$.  Let $(\nu,b)$ be the first play of such a $\tau$.  Thus, $\nu \in \Delta(X_j)$ and $(s,\nu) \in \text{E}_i$. Given $t \in b$ let $\nu(t)$ be the first play of the strategy $\tau_{t}$ obtained by following $\tau$ after the play $t$.  Note that $\tau_t$ is winning in $G_{t}^{E}$ by assumption.  We begin by setting $\delta^1_{i,\tau} = \nu$.  For $S \subseteq X_j \times \Delta(X_i)$ we define $\delta^2_{i,\tau}$ by $$\delta^2_{i,\tau}(S) = \delta^1_{i,\tau}( \{t \in b \ | \ (t,\delta_{j,\tau_{t}}^1)      \in S  \}  ) =\nu ( \{t \ | \ (t,\nu(t))           \in S  \}  ).$$  This is where measurability is used.  We define $\delta^3_{i,\tau}$ for $S \subseteq X_j \times \Delta(X_i) \times \Delta( X_i \times \Delta(X_j) )$ by $$\delta^3_{i,\tau}(S) = \delta^1_{i,\tau}(\{t \ | \ (t,\delta^1_{j,\tau_{t}},  \delta^2_{j,\tau_{t}}   )      \in S  \} ),$$
\noindent and in the general case 
$$\delta^n_{i,\tau}(S) = \delta^1_{i,\tau}(\{t \ | \ (t,\delta^1_{j,\tau_{t}},...,\delta^{n-1}_{j,\tau_{t}}   )      \in S  \} ) . $$

\noindent This defines a sequence of measures $t = (\delta^1_{i,\tau},..., \delta^n_{i,\tau},...)$ which can easily seen to satisfy common certainty of coherency (see Lemma 4.5 and Remark 4.7).  We need to show that $t$ witnesses $s \in \text{RCBR}(\text{E})_i$ (see Definition 3.3, Definition 4.6 and Lemma 4.8).  We have 
$$(s, \mbox{marg}_{X_j} \lambda_i(t) ) = (s,\delta^1_{i,\tau}) = (s, \nu)  \in \text{E}_i$$ 
as $\tau$ is a winning strategy.   We need to show by induction that 
$$\lambda_i(t)(\text{RCBR}(\text{E})_j^n) = 1$$
and by Lemma 4.8 it suffices to show that $$(s,\delta_{i,\tau}^1, ...,\delta_{i,\tau}^n) \in \text{RCBR}(\text{E})_i^{*,n}$$
\noindent for every $n \geq 2$.  By design and the fact that $\tau$ is winning, $\delta_{i,\tau}^2$ concentrates on the set $\{ (t,\nu(t)) \ | \ t \in b\}$ and hence on 
$\text{E}_j = \text{RCBR}(\text{E})_j^{*,1}$.  The induction step follows similarly.  \end{proof}

\begin{rem}  The above shows that measurable winning strategies give rise to types.  Thus Rationalizability is equivalent to the existence of a winning strategy, and RCBR follows from the existence of a measurable winning strategy. \end{rem}

\begin{cor}  Assume the Axiom of Real Determinacy $\mathsf{AD}_{\mathbb{R}}$ holds.  Then the following are equivalent for a strategy $s$ in a strategic form game with Polish strategy spaces and bounded payoff functions.
\begin{enumerate}
    \item The strategy $s$ is Rationalizable
    \item The strategy $s$ is consistent with RCBR
    \item Player $I$ has a winning strategy in the game $G^E_s$ where $E$ is the best response relation from the game.
\end{enumerate}
\end{cor}

\begin{proof}  This is the Fundamental Theorem of Epistemic Game Theory under Determinacy with the added Gale-Stewart condition (3).  Proposition 5.5 establishes the equivalence between (3) and (1).  The equivalence with (2) then follows from Theorem 5.6 as the measurability hypothesis is a consequence of determinacy, and also follows from Theorem 4.9.  \end{proof}

\section{Measurable Uniformizations}

The Jankov-von Neumann uniformization theorem (see 18.1 and 29.9 of \cite{kechris}) is used in \cite{lubin} to establish a measure extension property in the analytic case (see Lemma 2.11) which is central to the argument in \cite{arieli}.  Generalizing, let $A \subseteq X \times Y$ be an arbitrary set and let $\mu \in \Delta(X)$ with $\mu( \mbox{proj}_{X}(A)) = 1$.  If there is $$f: \mbox{proj}_X(A) \rightarrow Y$$ which is $\mu$-measurable and uniformizes $A$ then we can define $\nu(S)$ by $$\nu(S) = \mu ( \{ x \ | \ (x,f(x)) \in S \})$$ and conclude that $\nu \in \Delta(X \times Y)$ with $\nu(A) = 1$ and $\mbox{marg}_{X}(\nu) = \mu$.  Thus, the following axiom $\mathsf{MEA}$ is a consequence of $\mathsf{AD}_{\mathbb{R}}$.

\begin{defn}  Let $\mathsf{MEA}$ denote the following Measure Extension Axiom.  Suppose $X,Y$ are Polish spaces and $A \subseteq X \times Y$.  Suppose $\mu \in \Delta(X)$ and $$\mu( \mbox{proj}_{X}(A)) = 1.$$  Then there is $\nu \in \Delta(X \times Y)$ with $\nu(A) = 1$ and $\mbox{marg}_{X}(\nu) = \mu$.
\end{defn}

\begin{prop}  $\mathsf{MEA}$ is equivalent to the following proposition.  Suppose $X,Y$ are Polish spaces and $A \subseteq X \times Y$.  Suppose $\mu \in \Delta(X)$ and $$\mu( \mbox{proj}_{X}(A)) = 1.$$  Then there is a Borel set $B$ with $\mu(B) = 1$ and a measurable uniformization $f$ of $A$ defined on $B$, that is $(x,f(x)) \in A$ for every $x \in B$. \end{prop}

\begin{proof}  For the forward direction let $\nu \in \Delta(X \times Y)$ be the measure given by $\mathsf{MEA}$.  Let $B \subseteq A$ be Borel such that $\nu(B) = \nu(A)$.  By shrinking $B$ if necessary we can also assume $\mbox{proj}(B) = \bar{B}$ is Borel with $$\mu(\bar{B}) = \mu( \mbox{proj}_{X}(A)).$$  Thus there is a universally measurable uniformization $f$ of $B$ defined on $\bar{B}$ given by the Jankov-von Neumann Theorem.  For the other direction we can define $\nu$ from the uniformization $f$ by 
$$\nu(S) = \mu(  \{x \in B \ | \ (x,f(x)) \in S \})$$ for any Borel set $S \subseteq X \times Y$. \end{proof}

\begin{prop}  Assume $\mathsf{MEA}$.  Then all subsets of Polish spaces are universally measurable. \end{prop}

\begin{proof}  Assume $\mathsf{MEA}$ and let $S \subseteq X$ be a subset of a Polish space and let $\mu \in \Delta(X)$.  Let $A \subset X \times [0,1]$ be the graph of the characteristic function of $S$.  By Proposition 6.2 there is a Borel set $B \subseteq X$ with $\mu(B) = 1$ and a measurable uniformization $f$ of $A$ defined on $B$.  The sets $f^{-1}(1) = S \cap B$ and $f^{-1}(0) = (X \setminus S) \cap B$ are $\mu$-measurable.  We have 
$$B \cap S \subseteq S \subseteq X \setminus ( (X \setminus S) \cap B)$$ and $\mu(X \setminus B) = 0$ so $S$ is $\mu$-measurable.  \end{proof}

\begin{defn}  The Measurable Uniformization Principle $\mathsf{MUP}$ from the Solovay model (see \cite{solovay} and also \cite{rais}) asserts that for any family $$\{ A_x  \ | \ x \in B\}$$ of non-empty subsets of $\mathbb{R}$ indexed by a set $B \subseteq \mathbb{R}$ with $\mu(B) > 0$, where $\mu$ is Lebesgue measure, there is a Borel function $f$ such  that $$\mu(\{ x\in B \ | \ f(x) \notin A_x\})=0.$$
\end{defn}

\begin{prop}  $\mathsf{MEA}$ is equivalent to $\mathsf{MUP}$.
\end{prop}
\begin{proof}  For the forward direction, given $B$ as in the statement of $\mathsf{MUP}$, we can define the measure 
$$\lambda(S) = \frac{\mu(S \cap B)}{\mu(B)}$$ so that $\lambda(B) = 1$ and then apply Proposition 6.2.  Lusin's Theorem can then be used to extract a Borel function.

For the reverse direction, we are given $A \subseteq X \times Y$ and $\mu \in \Delta(X)$ with $$\mu( \mbox{proj}_{X}(A) ) = 1.$$  By Proposition 6.2 we need to produce a Borel set $B \subseteq X$ with $\mu(B) = 1$ and 
a measurable uniformization of $A$ defined on $B$.  We can assume that $X$ is uncountable.  Let $X_a$ be the set of atoms of $\mu$ and $X_c = X \setminus X_a$ so that $\mu$ is continuous on $X_c$.  The set $X_a$ is countable, hence Borel and we can uniformize on 
$\mbox{proj}_X(A) \cap X_a$, by simply picking elements.
Let $u_a:X_a \rightarrow Y$ be such a map with $(x,u_a(x)) \in A$ for $x \in X_a$.
So we may assume that $$\mu(\mbox{proj}_X(A) \cap X_c) > 0.$$ Viewing $X_c$ as a standard Borel space, there is a Borel isomorphism $\phi$ with the interval $[0, \mu(X_c)]$ equipped with Lebesgue measure (see 17.41 of \cite{kechris}).  We may also find a Borel injection $\psi$ of $Y$ into $[0,1]$.  We may then apply $\mathsf{MUP}$ to the set $$\hat{A} = \{ (\phi(x),\psi(y) \ | \ x \in X_c \wedge (x,y) \in A\} \subseteq [0, \mu(X_c)] \times [0,1]$$
\noindent to get an almost Borel funcion $\hat{u}$ which is almost everywhere a uniformization of $\hat{A}$.  Thus the map $$u = \psi^{-1} \circ \hat{u} \circ \phi$$ is the desired uniformization on $X_c$.  We can find a Borel set $B \subseteq X_c$ such that $\mu(B) = \mu(X_c)$ on which $u$ is defined.  Finally, we have $u_a \cup u$ is the desired uniformizing map on the Borel set $X_a \cup B$.  \end{proof}

\begin{thm} The Measure Extension Axiom $\mathsf{MEA}$ is equivalent to statement that $$\text{RAT}(\text{E}) = \text{RCBR}(\text{E})$$ for any Polish spaces $X_i$ and relations $\text{E}_i \subseteq X_i \times \Delta(X_j).$
\end{thm}

\begin{proof}  Let $A \subseteq X \times Y$ as in the statement of $\mathsf{MEA}$ with $\mu \in \Delta(X)$ such that $\mu(\mbox{prox}_X(A)) = 1$.  We want to produce a measure $\nu$ as in the statement of $\mathsf{MEA}$.  We can assume that $X,Y$ are uncountable Polish spaces.  Let $h:Y \rightarrow \Delta(X)$ be a Borel isomorphism and let $$\hat{A} = \{ (x,h(y)) \ | \ (x,y) \in A\} \subseteq X \times \Delta(X).$$  Note that $\mbox{proj}_X(A) = \mbox{proj}_{X}(\hat{A}).$  Let $$E_i = \hat{A} \cup ((X \setminus \mbox{proj}_{X}(\hat{A})) \times \Delta(X)).$$  

\noindent Note that if there is $\hat{\nu} \in \Delta(X \times \Delta(X))$ with $\hat{\nu}(E_i) = 1$ and $\mbox{marg}_X(\hat{\nu}) = \mu$ then $\hat{\nu}(\hat{A}) = 1$ and hence $\nu(A) = 1$ and $\mbox{marg}_{X}(\nu) = \mu$ where 
$$\nu(S) = \hat{\nu}(  \{ (x,h^{-1}(y)) \ | \ (x,y) \in S \} )$$ for any Borel $S \subset X \times Y$.  We regard $X = X_i = X_j$ as the strategy space for players $i$ and $j$, and have already defined $E_i \subseteq X_i \times \Delta(X_j)$ above.   We define $E_j \subseteq X_j \times \Delta(X_i)$ by $$E_j = X_j \times \{ \mu \}.$$ Thus, $$\text{RAT}(\text{E})_i = \text{RAT}(\text{E})_j = X.$$ By assumption, a given strategy in $X_j$ must belong to $\text{RCBR}(\text{E})_j$ as witnessed by a type $\delta_j$.  We must have $\delta_j^1 = \mu$, $\delta_j^2(\text{E}_i) = 1$, and  $$\mbox{marg}_{X_i}(\delta_j^2) = \mu.$$  Hence $\hat{\nu} = \delta_j^2$ is the desired measure.
 
For the other direction, we assume $\mathsf{MEA}$ holds and let $X_i$ be Polish and $$\text{E}_i \subseteq X_i \times \Delta(X_j)$$ be relations between strategies and beliefs.   For each i, let 
$$A^1_i = \{ (s,\mu) \in \text{E}_i \ | \ s \in \text{RAT}(\text{E})_i \mbox{ and } \mu(\text{RAT}(\text{E})_j) = 1\}$$
For $n>1$ and each $i$ define 
$$A^{n}_i = \{(s,\mu_i^1,...,\mu_i^{n}) \in  \text{RCBR}(\text{E})_i^{*,n} \ | \ $$$$(s,\mu_i^1,...,\mu_i^{n-1}) \in A^{n-1}_i \mbox{ and } \mu_i^{n}(A_j^{n-1}) = 1\}$$

\noindent We will show that every element of $A^{n}_i$ has an extension in $A^{n+1}_i$ by induction on $n$ with the convention that $A_i^{0} = \text{RAT}(\text{E})_i$.
Thus our induction hypothesis for $n \geq 0$ is 
$$A_{i}^{n} = \mbox{proj}(A_{i}^{n+1})$$
\noindent for each $i$.  Note that $A_{i}^{n} \subseteq \text{RCBR}(\text{E})_i^{*,n}$ by definition so this will prove that $\text{RAT}(\text{E})_i \subseteq \text{RCBR}(\text{E})_i$ and hence $\text{RAT}(\text{E})_i = \text{RCBR}(\text{E})_i$ for each $i$ as desired.  In fact, we need only require that $(s,\mu_i^1,...,\mu_i^{n})$ be coherent in the definition of $A_i^n$, and we would get $(s,\mu_i^1,...,\mu_i^{n}) \in \text{RCBR}(\text{E})_i^{*,n}$ from the rest of the definition. 

The case $n=0$ follows by definition of $\text{RAT}(\text{E})$.  Now let $\mu_i^1 \in \Delta(X_j)$ be such that $(s,\mu_i^1) \in A^1_i$.  Applying $\mathsf{MEA}$, let $\mu_i^2 \in \Delta(X_j \times \Delta(X_i))$ be such that $\mu_i^2(A^1_j) = 1$ and $$\mbox{marg}_{X_j}(\mu_i^2) = \mu^1_i.$$ Thus $(s,\mu_i^1,\mu_i^2) \in A_i^{2}$ as desired.  
For the case $n = 2$ let $(s,\mu_i^1,\mu_i^2) \in A_i^{2}$.  The key point is that $$A_j^1 = \mbox{proj}_{X_j \times \Delta(X_i)} A_j^2$$ by the induction hypothesis, so $\mu_i^2$ lifts to $\mu_i^3$ with $\mu_i^3(A_j^2) = 1$ and $$\mbox{marg}_{X_j \times \Delta(X_i)}(\mu_i^3) = \mu_i^2.$$  
Since $\mu_i^3$ concentrates on hereditarily coherent sequences (see Remark 4.7) we conclude that $$(s,\mu_i^1,\mu_i^2,\mu_i^3) \in A_i^3.$$  For the general case, let 
$$(s,\mu_i^1,\mu_i^2,...,\mu_n^i) \in A_i^n$$
We have $ A_j^{n-1}  =  \mbox{proj}_{X_i^{n-1}}(A_j^n)$ so $\mu_i^n$ lifts to $\mu_i^{n+1}$ with $\mu_i^{n+1}(A_j^{n}) = 1$ and $\mbox{marg}(\mu_i^{n+1}) = \mu_i^{n}$.  Again, $\mu_i^{n+1}$ concentrates on hereditarily coherent measures so we conclude that 
$$(s,\mu_i^1,\mu_i^2,...,\mu_n^i,\mu_i^{n+1}) \in A_i^{n+1}$$ by the observation of Remark 4.7  \end{proof}

\begin{thm}  The following are equivalent.
\begin{enumerate}
    \item The Measure Extension Axiom $\mathsf{MEA}$
    \item All sets Lebesgue Measurable plus  
      $$\text{RAT}(\text{E})^{*} = \text{RAT}(\text{E})$$ for any Polish spaces $X_i$ and relations $E_i \subseteq X_i \times \Delta(X_j)$
     \item $\text{RAT}(\text{E})^{*} = \text{RCBR}(\text{E}) = \text{RAT}(\text{E})$ for any Polish spaces $X_i$ and relations $E_i \subseteq X_i \times \Delta(X_j)$.
\end{enumerate}
\end{thm}

\begin{proof}  By Theorem 6.6 and Theorem 5.6 we have (2) implies (1).  By Proposition 6.3, Theorem 6.6 and Theorem 5.6 we see that (2) and (3) are equivalent.  So given Proposition 5.2, to finish the proof the theorem it suffice to turn a winning quasi-strategy in a game $G^{E}_s$ into a winning strategy under $\mathsf{MEA}$.  This will follow from Proposition 6.2.  

Suppose $\tau$ is a winning quasi-strategy for player $I$ in $G^{E}_s$ for some $s \in X_i$.  Let $(\mu_0,b_0)$ be a first play for I.  By Proposition 6.2 we may uniformize the response sets to $s_1 \in b_0$ to get $(\mu_1(s_1),b_1(s_1))$ as a function of $s_1$, and defined for all $s_1$ in some Borel $b^{*}_0 \subseteq b_0$ with $\mu_0(b^{*}_0) = 1$.  Now, let $m_1(s_1)$ and $b(s_1)$ be such a play.  We similarly apply Proposition 6.2 to uniformize the responses $\mu_2(s_2)$ for $s_2 \in b_1(s_1)$, shrinking $b_1(s_1)$ if necessary.   Continuing in this fashion by induction produces the desired strategy witnessing $s \in \text{RAT}(\text{E})^{*}_{i}$.    \end{proof}

\begin{cor}  Assume the Measure Extension Axiom $\mathsf{MEA}$.  Then Rationalizability is equivalent to Rationality and Common Belief in Rationality for any strategic form game with bounded payoffs played on a Polish strategy spaces.   Both properties are characterized by player I having a winning strategy in the auxiliary Gale-Stewart game.  \end{cor}

\begin{cor}  Suppose $G$ is a game played on Polish spaces with bounded payoffs, $G \in L(\mathbb{R})$ and $\mathsf{AD}$ holds in $L(\mathbb{R})$.  Then $$\text{RAT}(G) = \text{RCBR}(G).$$ 
\end{cor}
\begin{proof}  The principle $\mathsf{MUP}$ holds in $L(\mathbb{R})$, assuming $AD$ holds there, by Solovay reflection (see \cite{larson}).  So the result follows by Corollary 6.8 and Proposition 6.5, noting that the conclusion is absolute. \end{proof}

\section{The Analytic and Co-Analytic Cases}  Arieli proves in \cite{arieli} what can be regarded as the Fundamental Theorem of Epistemic Game Theory for Continuous Games.

\begin{theorem} (Arieli) Rationalizability is equivalent to Rationality and Common Belief in Rationality for any game with continuous bounded payoffs and Polish strategy spaces. \end{theorem}

To formulate the problem in the framework of interactive epistemics , let $X_i$ be Polish and $$\text{E}_i \subseteq X_i \times \Delta(X_j)$$ be analytic relations, for example the best response relation from a continuous game, which are closed.  It is easy to see by induction that the sets $\text{RAT}(\text{E})_i^{\alpha}$ are analytic for countable $\alpha$, so if $\alpha(\text{E})$ is countable for example, then $\text{RAT}(\text{E})$ itself is analytic, and we can use Lemma 2.11 as in the proof of Theorem 6.6 to directly construct types establishing the equivalence with $\text{RCBR}(\text{E})$.  The problem is that $$\alpha(\text{E}) = \omega_1$$ is a possibility (see section 5.5 of \cite{arieli} for an example).  The argument of \cite{arieli} does not yield a direct proof that $\text{RAT}(\text{E})$ is analytic, although this follows after the fact as $\text{RCBR}(\text{E})$ is analytic in this case.  Moreover, for a game with Borel measurable payoffs, the best response relation is generally co-analytic, 
$$\text{E}^{\text{G}}_i = (X_i \times \Delta(X_j)) \ \setminus $$$$\mbox{proj}_{X_i \times \Delta(X_j)} \{ (s,\mu,t) \in X_i \times \Delta(X_j) \times X_i \ | \ \pi_i(s,\mu) < \pi_i(t,\mu) \}.$$

In this section we will give a direct proof that $\text{RAT}(\text{E})$ is analytic for analytic $\text{E}$, and assuming Projective Determinacy, is projective for projective $\text{E}$.  This analysis will yield an extension of the fundamental theorem to a wider class of games in $\mathsf{ZFC}$ and to all Borel measurable (even analytically measurable) games under $\bPi^1_3$-Determinacy.

\begin{lem}  Suppose $A \subseteq X \times Y$ is $\bSigma^1_1$.  Then $$\{ (x,\mu) \in X \times \Delta(Y) \ | \ \mu(A_x) = 1\}$$ is $\bSigma^1_1$.   \end{lem}

\begin{proof}  We may assume that $A = \mbox{proj}_{X \times Y} (C)$ for some closed set $C \subseteq X \times Y \times Z$. 
The set $$\{ (x,\nu) \in X \times \Delta(Y \times Z) \ | \ 
\nu(C_x) = 1\}$$ is closed by standard arguments.
By the argument of Lemma 2.11 we know that there is $\mu$ with $\mu(A_x) = 1$ if and only if there is $\nu$ with $\nu(C_x) = 1$ and $\mbox{marg}_{Y}(\nu) = \mu$.  The desired set 
$$\mbox{proj}_{X \times \Delta(Y)}( \{ (x,\mu,\nu) \ | \ \nu(C_x) = 1 \wedge \mbox{marg}_{Y}(\nu) = \mu \})$$ is analytic.  Alternately by 17.25 of \cite{kechris} we observe that the map $$\phi:X \times \Delta(Y) \rightarrow \Delta(X \times Y)$$ define by $\phi(x,\mu) = \delta_x \times \mu$ is Borel and so $$\phi^{-1}( \{ \nu \in \Delta(X \times Y) \ | \ \nu(A) = 1 \} = \{ (x,\mu) \in X \times \Delta(Y) \ | \ \mu(A_x) = 1\}$$ is analytic as the inverse image of a set which is analytic (by the Lemma 2.11 argument).  \end{proof}

\begin{thm}  (Moschovakis, Kechris) Assume $\bPi^1_{2n+1}$-Determinacy for $n>0$.  Then for any Polish spaces $X,Y$, any $A \subseteq X \times Y$ which is $\bSigma^1_{2n+2}$ can be uniformized by an absolutely measurable function.  \end{thm}

\begin{proof} For $n>0$, it folllows from $\bDelta_{2n}^1$-Determinacy that the pointclasses $\bPi_{2n+1}^{1}$ and $\bSigma_{2n+2}^{1}$ have the uniformization property (see 6C.5 of \cite{mosch}).  Moreover, $\bPi_{2n+1}^1$-Determinacy implies that the pointclass $\bSigma_{2n+2}^{1}$ is universally measurable (see 6G.12 of \cite{mosch}).  For $n = 0$, the pointclass $\bSigma^1_2$ has the scale property (see 38.5 and 38.7 of \cite{kechris}) and so we can extract $\bSigma^1_2$ uniformizations from the scales.  Under the assumption of Analytic Determinacy, $\bSigma^1_2$ sets are measurable, and so these uniformizations are universally measurable.  \end{proof}

\begin{lem}  Assume $\bPi^1_{2n+1}$-Determinacy.  Given Polish spaces $X,Y$ and $$A \subseteq X \times Y$$ with $A \in \bSigma^1_{2n+2}$ and $\mu \in \Delta(X)$ with $$\mu( \mbox{proj}_{X}(A)) = 1,$$ there is $\nu \in \Delta(X \times Y)$ with $\nu(A) =1$ and $\mbox{marg}(\nu) = \mu$.
\end{lem}

\begin{proof}  By Theorem 7.2, there is a uniformization of $A$, $$f:\mbox{proj}_{X}(A) \rightarrow Y$$ which is $\bSigma^1_{2n+2}$ and universally measurable.  The push forward $\nu$ of $\mu$ by the map $g(x) = (x,f(x))$ to $X \times Y$ is the desired measure.   \end{proof}

\begin{rem}  Note that for $A \in \bSigma^1_3$ as in Lemma 7.3, the measure extension argument uses $\bSigma^1_4$ uniformizations as in Theorem 7.2, and so we need to assume $\bPi^1_3$-Determinacy. \end{rem}

\begin{lem}   Let $X,Y$ be Polish.  Let $\Gamma = \bSigma^1_n$ or $\Gamma = \bPi^1_n$.  Assume $\bPi^1_{n-1}$-determinacy if $n$ is even, and $\bPi^1_{n}$-determinacy if $n$ is odd.   If $A \subseteq X \times Y$ with $A \in \Gamma$, then  
$$\{(x,\mu) \in X \times \Delta(Y) \ | \ \mu(A_x) = 1\} \in \Gamma$$
and 
$$\{(x,\mu) \in X \times \Delta(Y) \ | \ \mu(A_x) > 0\} \in \Gamma.$$
\end{lem}

\begin{proof} We have established the property for $\bSigma^1_1$ in Lemma 7.1.  Thus, for $A \subseteq  X \times Y$ with $B \in \bPi^1_1$ the set 
$$\{ (x,\mu) \in X \times \Delta(Y) \ | \ \mu(A_x) > 0 \}$$ is $\bPi^1_1$, as the complement of such a set.  It is shown for example in Prop. 7.43 of \cite{shreve} that for any $\bSigma^1_1$ subset of a Polish space $A \subseteq X$, $$P(A) = \{ \mu \in \Delta(X) \ | \ \mu(A) > 0\}$$ is $\bSigma^1_1$.  It follows that if $A \subseteq X \times Y$ is $\bSigma^1_1$ then 
$$\{ (x,\mu) \in X \times \Delta(Y) \ | \ \mu(A_x) > 0\} = 
\phi^{-1}(P(A))$$ is $\bSigma^1_1$, where $$\phi: X \times \Delta(Y) \rightarrow \Delta(X \times Y)$$ is the Borel measurable map $\phi(x,\mu) = \delta_{x} \times \mu \in \Delta(X \times Y)$ (see 17.25 of \cite{kechris}).   It now follows by taking complements that if $B \subseteq X \times Y$ is $\bPi^1_1$ then $$\{(x,\mu) \in X \times \Delta(Y) \ | \ \mu(B_x) = 1\} \in \bPi^1_1.$$ This establishes the case $n=1$ with no determinacy assumption. 

In general, once we establish both properties for $\bSigma^1_n$ then we have them for $\bPi^1_n$ so it remains to prove the transfer from $\bPi^1_{n}$ to $\bSigma^1_{n+1}$ using Lemma 7.3.  Suppose $B \subseteq X \times Y$ with $B \in \bSigma^1_{n+1}$.  We want to show that the set 
$\hat{B} = \{ (x,\mu) \in X \times \Delta(Y) \ | \ \mu(B_x) > 0 \}$ is $\bSigma^1_{n+1}$.  Assume $B = \mbox{proj}_{X}(A)$ with $A \subseteq X \times Y \times Z$ and $A \in \bPi^1_{n}$.  We claim that $$\hat{B} = \mbox{proj}_{X \times \Delta(Y)} ( \{(x,\mu,\nu) \ | \ \nu \in \Delta(Y \times Z) \wedge \nu(A_x) > 0 \wedge \mu = \mbox{marg}(\nu)    \})$$

The point is if $(x,\mu) \in \hat{B}$ then we can let $f$ uniformize $A_x$ on $B_x$. Let $\bar{B}_x \subset Y$ be Borel with $\bar{B}_x \cap B_x = 0$ and $\mu(\bar{B}_x) = 1 - \mu(B_x)$.  We can extend $f$ to $B_x \cup \bar{B}_x$ by making $f$ constant on $\bar{B}_x$ and conclude that $f$ is $\bPi^1_n$ (or $\bPi^1_{n+1}$ is $n$ is even).  We can now let $\nu$ be the push forward of $\mu$ by the map $g(x) = (x,f(x))$ from $Y$ to $Y \times Z$.
It now follows that for $C \subseteq X \times Y$ in $\Pi^1_{n}$ that $\{ (x,\mu) \in Y \times \Delta(Y) \ | \ \mu(C_x) = 1\}$.
\end{proof}

\begin{rem}  In particular for $\bSigma^1_3$ we need $\bPi^1_3$-Determinacy, and for $\bSigma^1_2$ we need $\bPi^1_1$-Determinacy.  Note that as an immediate consequence, if the hypotheses apply and  $A \subseteq X$ is in $\Gamma$ then $$\{\mu \in \Delta(X) \ | \ \mu(X) = 1\}$$ is in $\Gamma$.  \end{rem}

\begin{prop}  Assume $\text{E}$ is a pair of relations between strategies and beliefs on underlying Polish spaces.  Then 
\begin{enumerate}

    \item If $\text{E}$ is $\bSigma^1_1$ then $$\text{RAT}(\text{E}) = \text{RCBR}(\text{E}) \Longleftrightarrow \text{ERAT is } \bSigma^1_1$$

    \item If $\bPi^1_1$-Determinacy holds and $\text{E}$ is $\bPi^1_1$ then $$\text{RAT}(\text{E}) = \text{RCBR}(\text{E}) \Longleftrightarrow \text{RAT}(\text{E}) \in \bSigma^1_2$$
     \item If $\bPi^1_{3}$-Determinacy holds and $\text{E}$ is $\bPi^1_{2}$ then $$\text{RAT}(\text{E}) = \text{RCBR}(\text{E}) \Longleftrightarrow \text{RAT}(\text{E}) \in \bSigma^1_{3}  .$$
    
\end{enumerate}
More generally, $\bPi^1_{2n+1}$-Determinacy implies that if $\text{RAT}(\text{E}) \in \bSigma^1_{2n+2}$ then $$\text{RAT}(\text{E}) = \text{RCBR}(\text{E}).$$
\end{prop}

\begin{proof}  The key point here is that Lemma 7.3 and Lemma 7.5 establish the properties of a pointclass $\Gamma = \bSigma^1_n$ required for the type construction argument of Theorem 6.6 under the assumption that $\text{RAT}(\text{E}) \in \Gamma$.  We need that  $A \in \Gamma$ implies 
$\{ \mu \ | \mu(A) = 1\} \in \Gamma$ and the extension property in 7.3.  
One simply folds into the induction hypothesis of the proof of Theorem 6.6 that the sets $A_i^n \in \Gamma$ for each $i$ and $n$.  Thus, the reverse implications above all hold, and more generally.
$\bPi^1_{2n+1}$-Determinacy implies that if $\text{RAT}(\text{E}) \in \bSigma^1_{2n+2}$ then  $$\text{RAT}(\text{E}) = \text{RCBR}(\text{E}).$$

The forward implications simply observe that we expect $\text{RAT}(\text{E})$ to have the same complexity as $\text{RCBR}(\text{E})$ if the fundamental theorem is to holds.  

If $\text{E}$ is $\bSigma^1_1$ then so is  $\text{RCBR}_i^{1}(\text{E})$ as the maps $$\mbox{marg}_{X_j}: \Delta(X_j \times T_j) \rightarrow \Delta(X_j)$$ and $\lambda_i$ are both continuous.  Thus each $\text{RCBR}_i^{n}(\text{E})$ is $\bSigma^1_1$ by Lemma 7.1, so each $\text{RCBR}_i(\text{E})$ is $\bSigma^1_1$ as the projection of the countable intersection of $\bSigma^1_1$ sets.    

For (2) we use the fact from Lemma 7.5 that 
for any $\bPi^1_1$ subset of a Polish space $B \subseteq X$, $$\{ \mu \in \Delta(X) \ | \ \mu(B) = 1)$$ is $\bPi^1_1$.  Again because the maps $\lambda_i$ and $\mbox{marg}_{X_j}$ are continuous, the set $\text{RCBR}(\text{E})_i^1$ is $\bPi^1_1$, and by induction the sets $\text{RCBR}(\text{E})^n_i(G)$ are $\bPi^1_1$ and so $\text{RCBR}(\text{E})$ is $\bSigma^1_2$ as the projection of the intersection of these sets.  Thus, $\text{RAT}(\text{E}) \in \bSigma^1_2$.

For (3) we assume $\bPi^1_{3}$-Determinacy and that $\text{E}$ is $\bPi^1_{2}$.  Again by lemma 7.5, the set of measures that concentrate on a given $\bPi^1_2$ set is itself $\bPi^1_2$.  Thus $\text{RCBR}(\text{E})$ is $\Sigma^1_{3}$ so $\text{RAT}(\text{E}) \in \bSigma^1_{3}$.  \end{proof}

\begin{thm}  Suppose $\text{E}$ is a pair of relations between strategies and beliefs on underlying Polish spaces.  

\begin{enumerate}
    \item If $\text{E}$ is $\bSigma^1_1$ then $\text{RAT}(\text{E})$ is $\bSigma^1_1$.  
    \item Assuming $\bPi^1_3$-determinacy, if $\text{E}$ is $\bPi^1_2$ then $\text{RAT}(\text{E})$ is $\bSigma^1_3$. 
\end{enumerate}

\noindent More generally, under $\bPi^1_{2n+1}$-determinacy, if $\text{E}$ is $\bPi^1_{2n}$ then $\text{RAT}(\text{E})$ is $\bSigma^1_{2n+1}$.

\end{thm}

\begin{proof}  We first treat the analytic case. For simplicity, we assume $X = X_1 = X_2$ and $\text{E} = \text{E}_1 = \text{E}_2$.  Thus $\text{RAT}_i^{\alpha}(\text{E}) = \text{RAT}_j^{\alpha}(\text{E})$ for every $\alpha$ so it becomes a one player problem.  Define the expansion operator  (in the terminology of Exercise 35.29 of \cite{kechris}) $$\Psi:P(X) \rightarrow P(X)$$  by $$\Psi(A) = \{ x \in X \ | \ x \in A \vee \neg \exists \mu \in \Delta(X) \ ((x,\mu) \in \text{E}_i \wedge \mu(X \setminus A) = 1) \}.$$  We set $\Gamma = \bPi^1_1$ and note that $\Gamma$ is a ranked class (see 34.1 of \cite{kechris}).

We claim that $\Psi$ is $\Gamma$ on $\Gamma$ in the sense of \cite{kechris}.  We have to show that if $Y$ is a Polish space and $A \subseteq Y \times X$ with $A \in \Gamma$ then $$A_{\Psi} = \{ (x,y) \ | \ x \in \Psi(A_y) \}$$ is also in $\Gamma$.  By definition, $A_{\Psi}$ is the set of $(x,y)$ such that $x \in A_y$ or there is no $\mu$ such that $(x,\mu) \in \text{E}_i$ and $\mu(X \setminus A_y) = 1$.  Thus, we may apply Lemma 7.1 to the set 
$B = \{(y,x) \ | \ (y,x) \notin A\}$ which is $\bSigma^1_1$ to
conclude that $$\hat{B} = \{ (y ,\mu) \ | \ \mu(B_y) = 1\} = \{ (y ,\mu) \ | \ \mu(X \setminus A_y) = 1\}$$ is $\bSigma^1_1$.
Thus $A_{\Psi} = \{(x,y) \ | \ (y,x) \in A \vee \neg \exists \mu ((x,\mu) \in \text{E}_i \wedge (y,\mu) \in \hat{B}) \}$ is in $\Gamma$.  Now let $U$ be $\Gamma$-universal for $2^{\omega} \times X$ and let 
$$P = \{ (q,x) \ | \ x \in \Psi( \{y \ | \ (q,y) <^{*}_{\psi}  (q,x)\} $$
where $<^{*}_{\psi}$ is from a $\Gamma$-rank $\psi$ (see 34.2 of \cite{kechris}).  To see that $P \in \Gamma$ we 
let $$A = \{ (q,x,z) \ | \ (q,z)  <^{*}_{\psi}  (q,x) \} \in \Gamma$$
so $(q,x,y) \in A_{\Psi}$ if and only if $y \in \Psi (\{ z \ | \ (q,z)  <^{*}_{\psi}  (q,x)  \}   )$.  Thus $A_{\Psi} \in \Gamma$ by the analysis above and so $$P = \{ (q,x) \ | \ (q,x,x) \in A_{\Psi} \} \in \Gamma.$$
By Kleene's Theorem (35.26 of \cite{kechris}; we may assume we chose $U$ to have the property) there is 
$p_0$ with $P_{p_0} = U_{p_0},$  
that is $$\{ x \ | \ (p_0,x) \in P \} = \{ x \ | \ (p_0,x) \in U \}.$$  It is now straightforward to check that $\Psi^{\infty}(\emptyset) = P_{p_0}$ which belongs to $\Gamma$.   But $$\Psi^{\infty}(\emptyset) = X_i \setminus \text{RAT}(\text{E})_i$$ so we conclude that $\text{RAT}(\text{E})_i$ is $\bSigma^1_1$.  To see this first note by induction on the rank $\gamma = \psi(p_0,x)$ of elements $x \in P_{p_0}$ that $x \in \Psi^{\gamma+1}(\emptyset) \subseteq \Psi^{\infty}(\emptyset)$.  Thus $P_{p_0} \subseteq \Psi^{\infty}(\emptyset)$.  Similarly by induction on $\beta$ we see that $\Psi^{\beta}(\emptyset) \subseteq P_{p_0}.$

Now, for the general case where the spaces $X_i$ and the relations $\text{E}_i$ are different, we work with the direct sum $X = X_1 \oplus X_2$, and define the expansion operator $\Psi:P(X) \rightarrow P(X)$ appropriately to handle both players at once, that is $\Psi(A)$ is the set of $x \in X$ such that $x \in A$ or 
$$(x \notin X_1 \vee \neg \exists \mu \in \Delta(X_2) \ ((x,\mu) \in \text{E}_1 \wedge \mu(X_2 \setminus A) = 1) ) \wedge $$$$(x \notin X_2 \vee \neg \exists \mu \in \Delta(X_1) \ ((x,\mu) \in \text{E}_2 \wedge \mu(X_1 \setminus A) = 1).$$
Both $X_1,X_2$ are closed in $X$ and the analysis is much the same as before.  We get $$U_{p_0} = \Psi^{\infty}(\emptyset) = X_1 \oplus X_2   \setminus \text{RAT}(\text{E})_1 \oplus \text{RAT}(\text{E})_2 \in \Gamma$$ so $\text{RAT}(\text{E})_1 \oplus \text{RAT}(\text{E})_2 \in \bSigma^1_1$ hence  $\text{RAT}(\text{E}) \in \bSigma^1_1$ as desired.

Finally for (2) let $\text{E}$ be $\bPi^1_2$.  In this case the associated expansion operator $\Psi$ is $\Gamma$ on $\Gamma$ for $\Gamma = \bPi^1_3$ which is a ranked class in the sense of \cite{kechris}.  For $\text{E} \in \bPi^1_1$ the  
operator is $\bPi^1_2$ which is not a ranked class.  The argument is as abovem, using Lemma 7.5 in place of Lemma 7.1.  We conclude that each $\text{RAT}(\text{E})_i$ is $\bSigma^1_3$.  \end{proof}

\begin{thm}  Suppose $\pi:X \times Y \rightarrow \mathbb{R}$ is bounded and Baire class 1, where $X,Y$ are locally compact Polish spaces.  Then the best response relation $$\text{E}^{\pi} = \{ (s,\mu) \in X \times \Delta(Y) \ | \ \forall t \in X \ \pi(s,\mu) \geq \pi(t,\mu) \}$$ is Borel. 
\end{thm}
\begin{proof}  The relation $\text{E}^{\pi}$ is clearly co-analytic as 
$$\text{E}^{\pi} = X \times \Delta(Y) \setminus \mbox{proj}_{X \times \Delta(Y)} \{ (s,\mu,t) \in X \times \Delta(Y) \times X \ | \ \pi(s,\mu) < \pi(t,\mu) \}$$
and the inequality $\pi(s,\mu) < \pi(t,\mu)$ is Borel for any Borel map $\pi$.  We will show it is analytic as well, and hence Borel by Suslin's Theorem.  Let $$\pi_n:X \times Y \rightarrow \mathbb{R}$$ be continuous with $\pi$ as the pointwise limit.  For $(s,\mu) \in X \times \Delta(Y)$ we will show that $(s,\mu) \in E^{\pi}$ if and only if there is an $\omega$-model $M$ of a sufficient fragment of set theory containing $\{ \pi_n \ | \ n<\omega\}$ as well as $(s,\mu)$ such that $(s,\mu) \in E^{\pi}$ holds in $M$.  This is an analytic assertion and so will complete the proof.  The forward direction is immediate so it suffices to show that if $M$ is such a model with $(s,\mu) \in M$ and 
$\pi(s,\mu) < \pi(t,\mu)$ for some $t \in X$, then $s$ is not a best response to $\mu$ in $M$.  Note that we may not assume that $t$ is in $M$.  Note also that each $\pi_n(x,\mu)$ is continuous as a function of $x$ and  converges to $\pi(x,\mu)$ for every $x \in X$.  Let $p$ be a rational with $$\pi(s,\mu) < p <  \pi(t,\mu)$$ and let $N$ be such that $p < \pi_n(t,\mu)$ for $n > N$.  Let $K \in M$ be the closure of a basic open neighborhood of $t$ which is compact.  This is where local compactness is used.  $M$ may not see $t$ but it sees such an $K$ as a countable base for the topology of $X$ lies in $M$.  Let $$K_n = \{ \bar{t} \in K \ | \ p \leq \pi_n(\bar{t},\mu) \}$$  as computed in $M$ and note that each $K_n$, for $n>N$, is nonempty and compact in $M$ by continuity.  Moreover, the sets $\{K_n \ | \ n>N\}$
have the finite intersection property in $M$, and hence there is $\bar{t} \in M$ with $p \leq \pi_n(\bar{t},\mu) $ for $n > N$.  Thus $$\pi(s,\mu) < p \leq \pi(\bar{t},\mu)$$ in $M$ so $M$ sees that $s$ is not a best response to $\mu$ as desired. \end{proof}

\begin{rem}  For a more direct proof of Theorem 7.4 we define 
$Q(\mu,p_1,p_2,p)$ by $$\forall N \ \exists \bar{N} > N \ \forall q \in [p_1,p_2] \ \exists n \in [N,\bar{N}] \ \pi_n(q,\mu) < p.$$
Then is is easy to check that $(s,\mu) \in E^{\pi}$ if and only if $$\forall p,p_1,p_2 \in \mathbb{Q} \ (p > \pi(s,\mu) \wedge (p_1 < p_2)) \Longrightarrow Q(\mu,p_1,p_2,p).$$
This asserts that in every compact neighborhood $[p_1,p_2]$, and every $p >\pi(s,\mu)$, no tail of the sequence of compact sets $$\{ x \ | \ x \in [p_1,p_2] \ \wedge \ \pi_n(x,\mu) \geq p \}$$ has the finite intersection property, which is easily seen to be equivalent to $$\neg \exists t \in \mathbb{R} \ \pi(s,\mu) < \pi(t,\mu).$$
Thus $E^{\pi}$ is Borel.

\end{rem}

\begin{thm}   Rationalizability is equivalent to Rationality and Common Belief in Rationality for any strategic form game with bounded Baire Class 1 payoffs played on a locally compact Polish strategy spaces. 
\end{thm}

\begin{proof}  By Theorem 7.8 the best response relation from such a game is Borel, and so the result follows from (1) of Proposition 7.2 together with (1) of Theorem 7.7.  \end{proof}

\begin{rem}  Both lower and upper semi-continuous functions are of Baire class 1.  In general, Baire class 1 functions are characterized by having a meager set of discontinuities (see \cite{oxtoby} or \cite{kechris}).  A Polish space is locally compact if and only if it is homeomorphic to open subspace of a compact Polish space.  \end{rem}

\begin{thm} (Arieli) Rationalizability is equivalent to Rationality and Common Belief in Rationality for any strategic form game with bounded and continuous payoffs played on Polish strategy spaces. 
\end{thm}

\begin{proof}  The best response relation from such a game is closed so the result follows from (1) of Proposition 7.6 together with (1) of Theorem 7.7.  \end{proof}

\begin{thm}   Assume $\bPi^1_3$-Determinacy.  Then Rationalizability is equivalent to Rationality and Common Belief in Rationality for any strategic form game with bounded Borel measurable (even anlytically measurable) payoffs and Polish strategy spaces. 
\end{thm}

\begin{proof}  In this case, the best response relation is co-analytic so the result follows from (3) of Theorem 7.6 together with part (2) from proposition 7.9.  In fact, if the payoffs from $\text{G}$ are analytically measurable (equivalently the payoffs have analytic graphs) then the best response relation $\text{E}^{\text{G}} \in \bPi^1_2$.  And the analysis of part (2) of Theorem 7.9 shows that $\text{RAT}(\text{E})$ is $\bSigma^1_3$ so the conclusion follows from part (3) of Theorem 7.6 as well.  \end{proof}

\begin{rem}  Note that by (2) of Proposition 7.6, we conclude that under $\bPi^1_3$-Determinacy, the set of rationalizable strategies in a Borel measurable game is $\bSigma^1_2$.  \end{rem}

\begin{thm}  Assume $\mathsf{PD}$.  Then Rationalizability is equivalent to Rationality and Common Belief in Rationality for any strategic form game with bounded projective payoffs and Polish strategy spaces. \end{thm}
\begin{proof}  For such a game $G$, there is $n$ such that $\text{E}^{\text{G}}$ is $\bPi^1_{2n}$. By Theorem 7.8 then $\text{RAT}(\text{E}) \in \bSigma^1_{2n+1}$.  The result follows from Proposition 7.7.
\end{proof}

Finally, we show that there is a Baire Class 2 payoff function on $\mathbb{R}^2$ with complete co-analytic best response relation.

\begin{prop}  There is a bounded Borel measurable function $$\pi: \mathbb{R} \times \mathbb{R} \rightarrow \mathbb{R}$$ with best response relation $$\text{E}^f = \{ (s,\mu) \in \mathbb{R} \times \Delta(\mathbb{R}) \ | \ \forall t \in X \ \pi(s,\mu) \geq \pi(t,\mu) \}$$ that is not Borel.  Moreover, $\pi$ can be Baire Class Two. \end{prop}

\begin{proof}  As observed in proof of Theorem 7.9, if $\pi$ is a Borel map on $X \times Y$ where each $X,Y$ are Polish,  then the best response relation $E^{\pi}$ is co-analytic. We give an example showing that $E^{\pi}$ can be complete analytic, and hence not Borel, for a Borel map $\pi$.
Let $\text{TR}$ be the Polish space of trees, regarded as a closed subspace of the Cantor space $2^{\omega^{<\omega}}$.  Let $\chi$ be the characteristic function of the set $$C = \{ (x,T) \in \omega^{\omega} \times \text{TR} \ | \ x \in [T] \}$$ where $[T]$ is the set of infinite branches of $T$.  The set $C$ is closed so $\chi$ is Baire class 1.  The best response relation - we look at pure strategies first - is $$E^{\chi} = \{ (x,T) \ | \ \forall y \in \omega^{\omega} \ \chi(x,T) \geq \chi(y,T) \} = (\omega^{\omega} \times \text{WF}) \cup C,$$ which is easily see to reduce $\text{WF}$.  To see this, given a tree $T \in \text{TR}$ we can map $T$ continuously to $(x_T',T')$ where $x_T' \notin [T']$ which gives the desired reduction.  Simply replace $T$ by $T' = \{ t' \ | \ t \in T \}$ where $t'(k) = t(k)+1$ for $k \in \mbox{dom}(t)$ and map $T$ to $(t_0,T')$ where $t_0$ is the constant $0$ branch.  Now we can use Borel isomorphisms $\phi: \omega^{\omega} \rightarrow \mathbb{R}$ and 
$\psi: \text{TR} \rightarrow \mathbb{R}$ and define $f$ on $\mathbb{R}^2$ by $$f(\psi(x),\phi(T)) = \chi(x,T).$$  Thus $f$ is Borel with complete co-analytic best-response set.  To see this, let 
$$E^{f} = \{ (x,\mu) \ | \ \forall y \in \mathbb{R} \ f(x,\mu) \geq f(y,\mu) \}$$
be the best response set.  We can form the map $$g: \text{TR} \rightarrow \mathbb{R} \times \Delta(\mathbb{R})$$
defined by $g(T) = (\phi(x),\delta_{\psi(T)})$ where $x \notin [T]$ and $\delta_{\psi(T)}$ is the measure concentrating on the singleton $\{ \psi(T) \}$.  As before $g(T) \in E^{f}$ if and only if $T \in \text{WF}$ (if $T \notin \text{WF}$ then $\phi(y)$ is a better response than $\phi(x)$ for any $y \in [T]$) so $E^{f}$ is complete coanalytic.  Note that the reduction map $g$ is Borel, which suffices by a result of Kechris \cite{kechris2}.  Thus, there is an obstruction to pushing the argument of Theorem 7.9 to all Borel games.  

Finally, we note that with a little more work we can make $f$ a Baire Class 2 map.  To see this note that $$\text{TR}^* = \{ T \in \text{TR} \ | \ \mbox{ T is infinite } \}$$ is homeomorphic to $\mathbb{R} \setminus \mathbb{Q}$ and so is $\omega^{\omega}$.  This uses the Brouwer characterization of the Cantor space (see 7.4 of \cite{kechris}) and the Alexandrov-Urysohn characterization of the Baire space (see 7.7 of \cite{kechris}).  Moreover $$C \subset \omega^{\omega} \times \text{TR}^{*}$$ so we can push $C$ to $$C' \subset (\mathbb{R} \setminus \mathbb{Q}) \times (\mathbb{R} \setminus \mathbb{Q})$$ via these homeomorphisms $\phi,\psi$: $$C' = \{(\phi(x),\psi(T)) \ | \ (x,T) \in C\}.$$  Thus $C'$ is a $G_{\delta}$ in $\mathbb{R} \times \mathbb{R}$ as the intersection of a closed set with a $G_{\delta}$ so 
its characteristic function $f$ is Baire class 2.  The best response relation $E^{f}$ is complete coanalytic as witnessed by the reduction given earlier.  \end{proof}

\begin{rem}  As mentioned in the introduction, results from \cite{cz} show that Iterated Dominance is more complex than Rationalizability in integer games (as the prototypical non-compact case).  We observe here that for a continuous game $\text{G}$ with Polish strategy spaces, the sets $\text{IU}_i(\text{G})$ are $\bPi^1_2$, in contrast to $\text{RAT}(\text{G}) = \text{RCBR}(\text{G})$ which is $\bSigma^1_1$ as we have seen.   The point is that $\bSigma^1_2$ is a ranked class.  The set $P$ of pairs $(q,x)$ such that $x$ is strictly dominated in $\text{G}$ on the set $$\{y \ | \ (q,y) <^{*}_{\psi} (q,x) \}$$ where $\psi$ is a $\Sigma^1_2$ norm on the universal $\bSigma^1_2$ set $U \subset 2^{\omega} \times X$, is $\bSigma^1_2$ on $\bSigma^1_2$ in the sense of \cite{kechris}.  Thus, we conclude that each $X_i \setminus \text{IU}(\text{G})_i$ is $\bSigma^1_2$ as in Theorem 7.8.  \end{rem}

\section{Counterexample Under The Continuum Hypothesis}

\begin{lem}  Let $X$ be a Polish space and $\mu \in \Delta(X)$.  Let $A$ be the set of isolated points of $X$. Then if there is a meager set $H$ with $\mu(H) = 1$ if and only if $\mu(A) = 0$.  \end{lem}

\begin{proof}  See also Lemma 4.1 of \cite{prikry} adapting Thm. 1.6 of \cite{oxtoby}.  If $\mu(A) > 0$ then no meager set can have full measure.  For the converse, let $Y = X \setminus A$ which is closed and hence Polish.  The atoms of $\mu$ must lie in Y, that is they are non-isolated points of $X$, so 
we can choose $\{ p_n \ | \ n<\omega \} \subset Y$ dense in $Y$ with $\mu(\{p_n\}) = 0$ for each $n < \omega$.
Now we can work the construction from Oxtoby.  Let $U_{n,k}$ be an open neighborhood of $p_{n}$ with $\mu(U_{n,k}) < 2^{-n-k-2}$.  The desired meager set is 
$$ H = \bigcup_{k<\omega} (Y \setminus \bigcup_{n<\omega} U_{n,k} )$$ which has measure one by design.\end{proof}

\begin{thm}  Assume the Continuum Hypothesis.  For any perfect Polish spaces $X_i$ there are universally measurable relations $\text{E}_i \subset X_i \times \Delta(X_j)$ such that $$\text{RAT}(\text{E}) \neq \text{RCBR}(\text{E}).$$  \end{thm}

\begin{proof}  For each $i$ let $H_{\alpha}^i$ enumerate the meager subsets of $\Delta(X_i)$.  Let $x_{\alpha}^i$ be an enumeration of $X_i$.  By induction we define $$f^i: X_i \rightarrow \Delta(X_j)$$ as follows.  Let $f^i(x_0^i) \in \Delta(X_j^0)$ be an atomless measure and for each $\alpha < \omega_1$ let 
$$f(x_{\alpha}^i) \notin \bigcup_{\beta < \alpha} H_{\beta}^j \cup \{ f^i(\beta) \}.$$  This is possible by the Baire Category Theorem.  We have 
$$\text{E}_i \cap (X_i \times H_{\alpha}^j)$$ is countable for every $\alpha < \omega_1$ where $\text{E}_i$ denotes the graph of $f^i$ for each $i$.
Thus $\mu( \text{E}_i) =0$ whenever $\mu \in \Delta(X_i \times \Delta(X_j))$ is an atomless measure so $\text{E}_i$ is a universally null and hence universally measurable subset of $X_i \times \Delta(X_j)$.  We now argue that $\text{E}$ is the desired counterexample.  First note that $\text{RAT}(\text{E})_i = X_i$, which follows from the fact that $f$ is a total map.  We claim that strategy $x = x^i_0$ is not consistent with $\text{RCBR}(\text{E})$.  Suppose to the contrary that there is a type $t(x)$ such that 
$$\mbox{marg}_{X_j}\lambda_i(t(x)) = f^i(x),$$ and 
$$\lambda_i(t(x)) ( \{ (y,t) \ | \ \mbox{marg}_{X_i} \lambda_j(t) = f^j(y) \} ) =  \lambda_i(t(x)) (\text{RCBR}(\text{E})_j^1   )   =  1.$$  

\noindent Let $\psi:X_j \times T_j \rightarrow X_j \times \Delta(X_i)$ be given by $$\psi(s,t) = (s, 
\mbox{marg}_{X_i} \lambda_j (t)),$$ which is continuous, and define $\nu \in \Delta(X_j \times \Delta(X_i))$ by $$\nu(S) = \lambda_i(t(x)) (\psi^{-1}(S))$$ for Borel sets $S \subseteq X_j \times \Delta(X_i).$  Then $\nu(E_j) = 1$ and $\mbox{marg}_{X_j}(\nu) = f^i(x)$.  The key point is that by Lemma 8.1 there is some $\alpha$ such that $$\nu(X_j \times H^i_{\alpha}) = 1.$$  Thus $\nu$ concentrates on a countable set, namely $$E_j \cap (X_j \times H^i_{\alpha}),$$ so $\nu$ must have an atom.  But then $f^i(x)$ must have an atom which is the desired  contradiction. \end{proof}

\begin{rem}  It is easy to see that if $\mathsf{V=L}$ holds then there is $\text{E}$ as above which is $\Delta^1_2$.  In contrast, Theorem 7.6 rules out such a counterexample for $\Sigma^1_1$ relations.  We do not know if there is a $\Pi^1_1$ counterexample.  \end{rem}

\begin{thm}  Assume the Continuum Hypothesis.  Then there is a game with Polish strategy spaces and with universally measurable best response relation that contains a rationalizable strategy which is not consistent with rationality and common belief in rationality.  \end{thm}
\begin{proof}  Let $I_0 = [0,1]$ and $I_1 = (1,2]$.  The strategy spaces will be $$X_1 = X_2 = X = [0,2] = [0,1] \cup (1,2] = I_0 \cup I_1.$$ Let $$a_n = 1 - \frac{1}{2^{n+1}}$$ \noindent for $n \geq 0$. Let $\mu$ denote Lebesgue measure.  We define $$\pi:[0,2] \times [0,2] \rightarrow \mathbb{R}$$ as follows.  For each $n$ and $I \subset I_0$ a basic open interval, select distinct points $x_{n,I} \in I_1$.  We will define $$\pi(x_{n,I},y) = a_n \pi_{I} (y)$$ for $y \in X$ where 

$$\pi_I(y) =
\begin{cases}
2- \mu(I)  &  y \in I\\
1- \mu(I) & y \notin I   \\
\end{cases}$$

\noindent For $x \in I_1$ that is not of the form $x_{n,I}$ we set $\pi(x,y) = 0$ for $y \in X$.  Now, given a map 
$$\phi: (0,1] \rightarrow [0,1]$$ we define $\pi(x,y)$ for $x \in I_0$ and any $y \in [0,2]$ by 

$$\pi(x,y) =
\begin{cases}
1  &  x = 0 \wedge y \in I_0 \\
0  &  x = 0 \wedge y \in I_1 \\
0 & x > 0 \wedge y \neq \phi(x)   \\
2 & x > 0 \wedge y = \phi(x)   \\
\end{cases}$$

\noindent This completes the definition of the payoff function $\pi$ on $X_1 \times X_2$ relative to the unspecified function $\phi$.  We define the payoff for the other player so that the game is symmetric.   We now determine which strategies are rationalizable.  Let $\mu_0 \in \Delta(I_0)$ denote Lebesque measure restricted to the unit interval.

\begin{claim}  The strategy $0$ is a best response to $\mu_0$, and to no other measure. 
\end{claim}
\begin{proof} First of all, we note that $\pi(0,
\mu_0) = 1$ and $\pi(x_{n,I},\mu_0) = a_n < 1$.  For any other $x$ we have $\pi(x,\mu_0) = 0$.  Now assume $\nu$ is a measure under which $0$ is a best response and $\nu \neq \mu_0$.  Without loss of generality, $\nu$ concentrates in $I_0$, that is, $\nu \in \Delta(I_0)$.   Assume first that $\nu$ is continuous.  There must be a basic basic open interval $I \subset I_0$ with $\mu_{0}(I) < 1$ such that $\mu_0(I) < \nu(I)$.  It follows that for $n$ sufficiently large $$1 = \pi(x_{n,I},\mu_0) < \pi(x_{n,I},\nu) = a_n (1 + \nu(I) - \mu(I)),$$ a contradiction.  Similarly, if $\nu$ has an atom $r \in I_0$ then we have $x_{n,I}$ is a better response to $\nu$ than $0$ for $I$ a sufficiently small neighborhood of $r$ and $n$ sufficiently large.  \end{proof}

\begin{claim}  Each strategy $x \in I_0$ with $x > 0$ is a best response to $\delta_{\phi(x)}$ and no other 
measure.
\end{claim}

\begin{proof}  We have $\pi(x,\delta_{\phi(x)}) = 2$ which is the global maximum value of $\pi$ so necessarily $x$ is a best response.  For any other measure $\nu$, that is $\nu \neq \delta_{\phi(x)}$, we necessarily have $\nu(\{\phi(x) \})<1$ so as $\phi(x) \in I_0$ we can find a $x_{n,I}$ with $\pi(x,\nu) < \pi(x_{n,I},\nu)$, by taking $I$ to be a basic open interval around $\phi(x)$ of sufficiently small measure and by taking $n$ sufficiently large.  \end{proof}

Finally, no strategy $x \in I_1$ is a best response strategy as each such strategy is strictly dominated.  As the game is symmetric we conclude that only strategies $x \in I_0$ are rationlizable, that is $$\text{RAT} =\text{RAT}_1 \times \text{RAT}_2 = I_0 \times I_0. $$  We now claim we can choose $\phi$ so that strategy $0$ is not consistent with $\text{RCBR}$.  Define $\phi$ by the same induction as in the proof of Theorem 8.2 using $I_0 = [0,1]$ in place of $\Delta(X)$.  Let $E_i$ be the best response relation for player $i$.  Thus $$E_i = \{ (0,\mu_0) \} \cup \{ (x, \delta_{\phi(x)}) \ | \ 0 < x \leq 1 \}$$ for each $i$.  Suppose toward a contradiction that there is $\nu \in \Delta( X \times \Delta(X))$ such that $\nu(E_j) = 1$ and $$\mbox{marg}_{X}(\nu) = \mu_{0}$$  Then we must have $$\nu ( \{ (x,\delta_{\phi(x)})  \ | \ x > 0 \}) = 1,$$ and thus $\nu$ gives rise to a continuous measure $\hat{\nu} \in \Delta(X \times X)$ with $$\hat{\nu}(  \{ (x,\phi(x))  \ | \ x > 0 \} ) = 1$$ 
and $\mbox{marg}_X( \hat{\nu}) = \mu_0$, a contradiction as in the last theorem.
\end{proof}

\begin{thm}  Assume Rationalizability is equivalent to Rationality and Commmon Belief in Rationality for all games with bounded payoffs and Polish strategies spaces.  Then all sets are Lebesgue measurable.
\end{thm}
\begin{proof}  Let $A \subseteq [0,1]$ be an arbitrary set of reals.  Define a symmetric game as in Theorem 8.4 with payoff $\pi^A$ satisfying 

$$\pi^A(x,y) =
\begin{cases}
2  &  x \in  A  \wedge y = 1 \\
0  &  x \in  A \wedge y \neq 1 \\

1  &  x \in (I_0 \setminus A) \wedge y \in I_0 \\
0 & x \in (I_0 \setminus A) \wedge y \notin I_0    \\
\end{cases}$$

\noindent for $x \in I_0$, and $y \in I_0 \cup I_1 = [0,2]$, and exactly the same as $\pi$ from Theorem 8.4 for $x \in I_1$.  In other words, we use $\phi(x) = 1$ for $x \in A$, in the definition from Theorem 8.4.  Using the analysis there we see that 
$$E_i = \{ (x,\mu_0) \ | \ x \in (I_0 \setminus A) \} \cup 
\{(x, \delta_{1}) \ | \ x \in A \}$$ for each $i$ where $\mu_0$ is Lebesgue measure on $I_0$, and that $\text{RAT}(\text{E}) = I_0 \times I_0$.  Thus since $0$ is consistent with $\text{RCBR}$, there must be a measure 
$$\nu \in \Delta(I_0 \times \Delta(I_0))$$ with $\mbox{marg}_{I_0} (\nu) = \mu_0$ and $\nu(E_j) = 1$ as before.  It follows that $A$ is $\mu_0$-measurable as in Proposition 6.3.  There are Borel sets $B_0 \subseteq 
(I_0 \setminus A) \times \{ \mu_0\}$ and $B_1 \subset A \times \{\delta_1\}$ and such that $\nu(B_0) + \nu(B_1) = 1$.  Hence 
$$m_0(\mbox{proj}_{I_0}(B_0)) + m_0(\mbox{proj}_{I_0}(B_1)) = 1$$ from which we conclude that $A$ is $\mu_0$ measurable.
\end{proof}

\begin{thm}  The statement that rationalizable strategies are consistent with rationality and common belief in rationality in any game with bounded payoffs and Polish strategy spaces is equiconsistent with the existence of an inaccessible cardinal. \end{thm}

\begin{proof}  One direction follows from Theorem 8.7 and a Theorem of Shelah (see \cite{kanamori}).  The other directions from Proposition 6.5 and Corollary 6.8, noting that $\mathsf{MUP}$ holds in the Solovay model (see \cite{solovay}). \end{proof}

%\end{multicols}

\end{document}